\documentclass[11pt, fleqn]{article}
\usepackage[margin=1in]{geometry}
\usepackage[T1]{fontenc}
\usepackage[utf8]{inputenc}
\usepackage[noadjust]{cite}
\usepackage[all=normal,bibliography=tight]{savetrees}
\usepackage{tikz}
\usetikzlibrary{patterns}
\usetikzlibrary{decorations.pathmorphing}
\tikzset{snake it/.style={decorate, decoration=snake}}
\usepackage{listings}
\usepackage{cite}
\usepackage{mathrsfs}
\usepackage{amstext,amsfonts,amsthm,amsmath,amssymb}
\usepackage{graphicx,tikz}
\usepackage{comment}
\usepackage{url}
\usepackage{xspace}
\usepackage{todonotes}
\usepackage[shortlabels]{enumitem}
\usepackage[absolute]{textpos}
\usepackage{thm-restate}
\usepackage{subcaption}

\def\longbox#1{\parbox{0.85\textwidth}{#1}}

\makeatletter
\newcommand{\leqnomode}{\tagsleft@true}
\newcommand{\reqnomode}{\tagsleft@false}
\makeatother

\def\cqedsymbol{\ifmmode$\lrcorner$\else{\unskip\nobreak\hfil
\penalty50\hskip1em\smallOll\nobreak\hfil$\lrcorner$
\parfillskip=0pt\finalhyphendemerits=0\endgraf}\fi} \newcommand\smallO{
  \mathchoice
    {{\scriptstyle\mathcal{O}}}
    {{\scriptstyle\mathcal{O}}}
    {{\scriptscriptstyle\mathcal{O}}}
    {\scalebox{.7}{$\scriptscriptstyle\mathcal{O}$}}
  }

\newtheorem{lemma}{Lemma}[section]

\newtheorem{theorem}[lemma]{Theorem}

\theoremstyle{definition}

\def\dd{\hbox{-}}

\renewcommand{\S}{\mathcal{S}}

\newcommand{\R}{\mathcal{R}}

\newcommand{\Pre}{\text{Pre}}

\def\dd{\hbox{-}}

\newcommand*\samethanks[1][\value{footnote}]{\footnotemark[#1]}

\title{Submodular functions and perfect graphs}
\author{
Tara Abrishami$^{1}$\thanks{Supported by NSF Grant DMS-1763817 and NSF-EPSRC Grant DMS-2120644} \ \ \ \
Maria Chudnovsky$^{1}$\samethanks \ \ \ \
Cemil Dibek$^{1}$\thanks{Supported by NSF Grant DMS-1763817.} \ \ \ \
Kristina Vu\v{s}kovi\'c$^{2}$\thanks{Supported by DMS-EPSRC grant EP/V002813/1.}\\\\
 $^1${\small Princeton University, Princeton, NJ, USA}\\
\small $^2$School of Computing, University of Leeds, UK
\\
\\
}
\date{\today}

\begin{document}\maketitle
\vspace{-.3cm}
\begin{abstract}
  We give a combinatorial polynomial-time algorithm to find a maximum weight independent set in perfect graphs of bounded degree that do not contain a prism or a hole of length four as an induced subgraph. An {\em even pair} in a graph is a pair of vertices all induced paths between which are even. An {\em even set} is a set of vertices every two of which are an even pair. We show that every perfect graph that does not  contain a prism or a hole of length four as an induced subgraph has a balanced separator which is the union of a bounded number of even sets, where the bound depends only on the maximum degree of the graph. This allows us to solve the maximum weight independent set problem using the well-known submodular function minimization algorithm.

\end{abstract}

\section{Introduction}\label{sec:intro}
All graphs in this paper are finite and simple. For two graphs $G$ and $H$, we say that $G$ \emph{contains} $H$ if some induced subgraph of $G$ is isomorphic to $H$. A graph $G$ is \emph{$H$-free} if it does not contain $H$, and when $\mathcal{H}$ is a set of graphs, we say $G$ is $\mathcal{H}$-free if it is $H$-free for all $H$ in $\mathcal{H}$.

A \emph{clique} in a graph is a set of pairwise adjacent vertices, and an \emph{independent set} is a set of pairwise non-adjacent vertices. The \emph{chromatic number} of a graph $G$ is the smallest number of independent sets of $G$ with union $V(G)$. A graph $G$ is \emph{perfect} if every induced subgraph $H$ of $G$ satisfies $\chi(H) = \omega(H)$, where $\chi(H)$ is the chromatic number of $H$ and $\omega(H)$ is the size of a maximum clique in $H$. For an integer $k \geq 4$, a {\em hole of length $k$} in a graph is an induced subgraph isomorphic to the $k$-vertex cycle $C_k$. An \emph{antihole} is the complement of a hole. A hole or antihole is {\em odd} if its length is odd, and {\em even} if its length is even. A graph is \emph{Berge} if it does not contain an odd hole or an odd antihole. Claude Berge introduced the class of perfect graphs and conjectured that a graph is perfect if and only if it is Berge. This conjecture (now the Strong Perfect Graph Theorem) was proved by Chudnovsky, Robertson, Seymour and Thomas \cite{SPGT}.

Given a graph with non-negative weights on its vertices, \textsc{Maximum Weight Independent Set} (MWIS) is the problem of finding an independent set of maximum total weight. It is known that MWIS can be solved in polynomial time in perfect graphs due to the algorithm of Gr\"{o}tschel, Lov\'asz and Schrijver \cite{Grotschel}. This algorithm, however, uses the ellipsoid method to solve semidefinite programs. Although there is no standard definition of {\em combinatorial} algorithm, most graph theorists agree that algorithms that rely on graph searches and decompositions and that can be described as a sequence of operations applied directly to the vertices and edges of the graph can be called combinatorial. Under this ``definition'' the algorithm of \cite{Grotschel} is not considered to be combinatorial. Currently, no combinatorial polynomial-time algorithm is known to solve MWIS in perfect graphs, and consequently there has been interest in studying MWIS in restricted subclasses of perfect graphs. Although in some restricted subclasses of perfect graphs, MWIS can be formulated as a linear program of polynomial size and different approaches can then be used to solve it, translating these methods into a ``combinatorial algorithm'' is still out of reach, so even finding a combinatorial polynomial-time algorithm for MWIS in subclasses of perfect graphs is an open and interesting problem. In this paper, we prove one such result. A \emph{prism} is a graph consisting of two vertex-disjoint triangles $\{a_1, a_2, a_3\}$, $\{b_1, b_2, b_3\}$, and three vertex-disjoint paths $P_1, P_2, P_3$, where each $P_i$ has endpoints $a_i, b_i$, and for $1 \leq i < j \leq 3$ the only edges between $V(P_i)$ and $V(P_j)$ are $a_ia_j$ and $b_ib_j$. In this paper, we obtain a combinatorial polynomial-time algorithm solving MWIS in perfect graphs of bounded degree that do not contain a prism or a hole of length four. Our approach uses two general tools, even set separators and iterated decompositions, which can be applied to  different graph classes and which therefore add to the available methods to solve MWIS. Developing the theory of even set separators and of iterated decompositions are major contributions of this paper, aside from the application to MWIS in a subclass of perfect graphs. Next, we describe how these two tools work and why they are useful.

Even set separators are related to a well-known graph structure called balanced separators. A balanced separator in a graph is a set of vertices that, when deleted, breaks the graph into small components. When a balanced separator is of constant size, algorithmic problems can be solved in the graph in polynomial time using recursion. For instance, in the case of MWIS, one can guess the intersection of a maximum independent set with the balanced separator in constant time when the size of the balanced separator is constant, and then compute MWIS in the components recursively since the components are small. Even set separators similarly have the property that, when deleted, the remaining components of the graph are small. However, instead of relying on constant size to make MWIS solvable in polynomial time, even set separators allow  us to reduce an instance of MWIS to several instances of submodular function minimization, a problem known to be solvable in polynomial time \cite{IFF, McCormick, Schrijver, O-Sbm}.
  
More formally, let $G$ be a graph and let $w: V(G) \to [0, 1]$ be a weight function defined on the vertices of $G$. For $X \subseteq V(G)$, let $w(X) = \sum_{x \in X} w(x)$, and denote $w(V(G))$ by $w(G)$. We let $w^{\max}$ be the maximum weight of a vertex, i.e. $w^{\max} = \max_{v \in V(G)} w(v)$. A weight function $w$ on $V(G)$ is called a {\em uniform weight function} if there exists $X \subseteq V(G)$ such that $w(v) = \frac{1}{|X|}$ if $v \in X$ and $w(v) = 0$ if $v \not \in X$. Let $c \in [\frac{1}{2}, 1)$. A set $X \subseteq V(G)$ is a {\em $(w, c)$-balanced separator} if every connected component $D$ of $G \setminus X$ satisfies $w(D) \leq c$. Balanced separators of ``small'' size are useful because they allow many algorithmic problems to be solved in polynomial time using recursion. For example, in graphs that have a small balanced separator for every uniform weight function, \textsc{MWIS} is solvable in polynomial time.

\begin{lemma}[\cite{HW}]
  \label{lemma:bounded_balanced_gives_IS}
  There is a function $f: \mathbb{N} \times \mathbb{R} \rightarrow \mathbb{N}$
  with the following property.
Let $G$ be a graph with $|V(G)| = n$. Let $c \in [\frac{1}{2}, 1)$, let $k$ be a nonnegative integer, and suppose $G$ has a $(w, c)$-balanced separator of size at most $k$ for every uniform weight function $w$. Then, MWIS can be solved in $G$ in time at most $n^{f(k,c)}$. 
\end{lemma}

A {\em path} in $G$ is an induced subgraph isomorphic to a graph $P$ with vertices $p_0, p_1, \dots, p_k$ and with $E(P) = \{p_ip_{i+1} : i \in \{0,\dots, k-1\}\}$. We write $P = p_0 \dd p_1 \dd \dots \dd p_k$ to denote a path with vertices $p_0, p_1, \dots, p_k$ in order. We say that $P$ is a path from $p_0$ to $p_k$. The \emph{length} of a path $P$ is the number of edges in $P$. A path is {\em odd} if its length is odd, and {\em even} otherwise. For a path $P$ with ends $a, b$, the \emph{interior} of $P$, denoted $P^*$, is the set $V(P) \setminus \{a, b\}$.

The {\em distance} between two vertices $x, y \in V(G)$ is the length of the shortest path from $x$ to $y$ in $G$. The distance between a vertex $v \in V(G)$ and a set $X \subseteq V(G)$ is the length of the shortest path with one end $v$ and the other end in $X$. We denote by $N^d[v]$ the set of all vertices of distance at most $d$ from $v$ in $G$. Similarly, we denote by $N^d[X]$ the set of all vertices of distance at most $d$ from $X$ in $G$.

Since we will be focusing on graphs with bounded maximum degree, we use an alternative definition of bounded.  We say that a set $X \subseteq V(G)$ is {\em $d$-bounded} if $X \subseteq N^d[v]$ for some $v \in V(G)$. A set $X \subseteq V(G)$ is a {\em $d$-bounded $(w, c)$-balanced separator} of $G$ if $X$ is $d$-bounded and $X$ is a $(w, c)$-balanced separator of $G$. Note that if $G$ has maximum degree $\delta$ and if $X$ is $d$-bounded, then $|X| \leq 1 + \delta + \hdots + \delta^d$.

In this paper, we define a new type of separator called even set separators. An \emph{even pair} in $G$ is a pair of vertices $\{x, y\}$ such that every induced path in $G$ from $x$ to $y$ is even, and in particular, $x$ and $y$ are non-adjacent. A set $X \subseteq V(G)$ is an {\em even set} if every pair of its vertices is even. Note that all even sets are independent sets. Let $X_1,\hdots, X_\ell$ be pairwise disjoint vertex subsets of $G$. We say that $(X_1, \hdots, X_\ell)$ is an {\em $\ell$-iterated even set in $G$} if $X_i$ is an even set in $G \setminus (\bigcup_{j < i} X_j)$ for all $1 \leq i \leq \ell$.

Let $k$ and $d$ be positive integers and let $c \in [\frac{1}{2}, 1)$. We say that $X = (X_1, \hdots, X_k)$ is a {\em $(w, k, c, d)$-even set separator of $G$} if $X$ is a $k$-iterated even set in $G$ and every connected component $D$ of $G \setminus X$ satisfies $|N(D)| \leq d$ and $w(D) \leq c$ (here, we use $G \setminus X$ to mean $G \setminus (\bigcup_{i=1}^k X_i))$. Note that a $(w, c)$-balanced separator $X=\{x_1, \hdots, x_k\}$ of size at most $k \leq d$ is a $(w, k, c, d)$-even set separator, since $(\{x_1\}, \hdots, \{x_k\})$ is an iterated even set. A graph $G$ is called {\em $(k, c, d, m)$-tame} if $G$ has a $(w, k, c, d)$-even set separator for every uniform weight function $w$, and such a separator can be constructed in $\mathcal{O}(|V(G)|^m)$ time.

In Section \ref{sec:even_set_separators}, we prove the following theorem.
\begin{restatable}{theorem}{evensetseparator}
\label{thm:evensetseparator}
Let $k,d,m$ be integers, and let $c \in (\frac{1}{2},1)$. Let $z$ be the minimum integer such that $c^{\frac{z-1}{d+1}} \leq \frac{1}{2}$. 
Then, there is a combinatorial algorithm that, given a $(k,c,d,m)$-tame graph $G$, solves MWIS in $G$ in ${\cal O}(|V(G)|^y)$-time,
where $y=\max(m, z, 5k+1)$.
\end{restatable}

The proof of Theorem \ref{thm:evensetseparator} relies on the well-known submodular function minimization algorithm. In this paper, we use Theorem \ref{thm:evensetseparator} to solve MWIS in ($C_4$, prism)-free perfect graphs of bounded degree. In fact, we work with a slightly larger class of ``paw-friendly graphs'', that we define in Section \ref{sec:star_free_bag}. We remark that this class contains graphs with arbitrarily large treewidth (e.g., bipartite subdivisions of a large wall), and therefore the standard techniques such as applying dynamic programming over the decomposition tree to solve MWIS do not work for graphs that we consider in this paper.

In Section \ref{sec:lastsection}, we prove the following theorem.

\begin{restatable}{theorem}{pawfriendlyevenseparator}
\label{thm:pawfriendlyevenseparator}
Let $c \in [\frac{1}{2}, 1)$ and let $\delta$ be a positive integer. Let $G$ be a paw-friendly  graph with maximum degree $\delta$. Then, $G$ is $(1 + \delta + \hdots + \delta^{\delta + 3}, c, 1 + \delta + \hdots + \delta^{\delta + 3}, 3)$-tame.
  \end{restatable}

Together, Theorems \ref{thm:evensetseparator} and \ref{thm:pawfriendlyevenseparator} imply that there is a combinatorial ploynomial-time algorithm to solve MWIS in paw-friendly graphs with bounded degree.

\begin{theorem}
  \label{thm:MWIS}
  There is a function $f: \mathbb{N} \rightarrow \mathbb{N}$ with the following
  property.
There is a combinatorial algorithm that, given a paw-friendly graph $G$  on $n$ vertices with maximum degree $\delta$, solves MWIS in $G$
in time at most $n^{f(\delta)}$.
\end{theorem}

Since by Theorem~\ref{thm:C4prismpawfriendly} every ($C_4$, prism)-free perfect graph
is paw-friendly, we deduce:

\begin{theorem}
\label{thm:c4prismevenseparator}
Let $c \in [\frac{1}{2}, 1)$ and let $\delta$ be a positive integer. Let $G$ be a ($C_4$, prism)-free perfect graph with maximum degree $\delta$. Then, $G$ is $(1 + \delta + \hdots + \delta^{\delta + 3}, c, 1 + \delta + \hdots + \delta^{\delta + 3}, 3)$-tame.
  \end{theorem}

\begin{theorem}
  \label{thm:MWISC4prism}
  There is a function $f: \mathbb{N} \rightarrow \mathbb{N}$ with the following
  property.
There is a combinatorial algorithm that, given a ($C_4$, prism)-free perfect graph $G$  on $n$ vertices with maximum degree $\delta$, solves MWIS in $G$
in time at most $n^{f(\delta)}$.
\end{theorem}

Theorems \ref{thm:c4prismevenseparator} and \ref{thm:MWISC4prism} are two distinct contributions related to ($C_4$, prism)-free perfect graphs with bounded degree. The first is a structure theorem proving the presence of even set separators. The second is an application of submodularity to solve MWIS. The remainder of the paper is structured as follows. In Section \ref{sec:even_set_separators}, we define submodular functions and prove Theorem \ref{thm:evensetseparator}. In Sections \ref{sec:central_bags} and \ref{sec:star_separations}, we define iterated decompositions and star separations, and describe their key properties. In Section \ref{sec:paws_are_forcers}, we prove that the presence of certain induced subgraphs in ($C_4$, prism)-free perfect graphs forces a decomposition. In Section \ref{sec:star_free_bag}, we describe iterated decompositions in ($C_4$, prism)-free perfect graphs. In Section \ref{sec:R_evensetseparator}, we construct iterated even sets in ($C_4$, prism)-free perfect graphs. Finally, in Section \ref{sec:lastsection}, we prove Theorem \ref{thm:pawfriendlyevenseparator}. 

\paragraph{Definitions and notation.}
Let $G = (V, E)$ be a graph. For $X \subseteq V(G)$, $G[X]$ denotes the induced subgraph of $G$ with vertex set $X$ and $G \setminus X$ denotes the induced subgraph of $G$ with vertex set $V(G) \setminus X$. We use induced subgraphs and their vertex sets interchangeably throughout the paper. For a graph $H$, we say that a set $X \subseteq V(G)$ is an $H$-copy in $G$ if $G[X]$ is isomorphic to $H$.

Let $X \subseteq V(G)$. The \emph{neighborhood} of $X$ in $G$, denoted by $N(X)$, is the set of all vertices in $V(G) \setminus X$ with a neighbor in $X$. The \emph{closed neighborhood} of $X$ in $G$, denoted $N[X]$, is given by $N[X] = N(X) \cup X$. For $u \in V(G)$, $N(u) = N(\{u\})$ and $N[u] = N[\{u\}]$. For $u \in V(G) \setminus X$, $N_X(u) = N(u) \cap X$. Let $Y \subseteq V(G)$ be disjoint from $X$. We say $X$ is \emph{complete} to $Y$ if every vertex in $X$ is adjacent to every vertex in $Y$, and $X$ is \emph{anticomplete} to $Y$ if every vertex in $X$ is non-adjacent to every vertex in $Y$. Note that the empty set is complete and anticomplete to every $X \subseteq V(G)$. We say that a vertex $v$ is \emph{complete} (\emph{anticomplete}) to $X \subseteq V(G)$ if $\{v\}$ is complete (anticomplete) to $X$. A \emph{cutset} of $G$ is a subset $K \subseteq V(G)$ such that $G \setminus K$ is not connected. A set $S \subseteq V(G)$ is a \emph{star cutset} of $G$ if $S$ is a cutset and there exists $v \in S$ such that $S \subseteq N[v]$.

\section{Submodular functions and even set separators}
\label{sec:even_set_separators}

In this section, we describe a combinatorial algorithm to solve MWIS in $(k, c, d, m)$-tame graphs that runs in time polynomial in the number of vertices (with $k, c, d, m$ fixed). To do so, we make use of submodular functions. Given a finite set $S$, a set function $f:2^S \rightarrow \mathbb{R}$ is said to be \emph{submodular} if
$$f(A) + f(B) \geq f(A \cup B) + f(A \cap B),$$
for all subsets $A, B$ of $S$. The above inequality is known as the \emph{submodular inequality}. There are several examples of submodular functions that appear in graph theory, see \cite{L-Subm} for an exposition of these examples.

Minimizing a submodular function $f:2^S \rightarrow \mathbb{R}$ is the problem of finding a subset of $S$ that minimizes $f$. Assuming the availability of an \emph{evaluation oracle}, that is, a black box whose input is some set $U \subseteq S$, and whose output is $f(U)$, there exist fully combinatorial strongly polynomial time algorithms for submodular function minimization, see \cite{McCormick} for a survey of these algorithms. These are algorithms which use only additions, subtractions, and comparisons, and whose running time is a polynomial function of $|S|$ only. (In contrast, submodular function maximization is known to be NP-hard.)

\begin{lemma}[\cite{O-Sbm}]
Let $f:2^S \rightarrow \mathbb{R}$ be a submodular function defined on the subsets of a set $S$ with $n$ elements. Then, there is a combinatorial algorithm to minimize $f$ in time $\mathcal{O}(n^5EO + n^6)$, where EO is the running time of evaluating $f(A)$ for a given $A \subseteq S$.
\label{lemma:MWIS_submodular}
\end{lemma}

Let $G$ be a graph with vertex weight function $w$. Let $\alpha(G)$ denote the maximum weight of an independent set of $G$. Let $S$ be an independent set in a graph $G$ and let $A \subseteq S$. We let $I_{A,S}(G)$ denote a maximum weight independent set $I$ in $G$ such that $I \cap S = A$. In words, $I_{A,S}(G)$ is a maximum weight independent set in $G$ that ``extends'' the set $A \subseteq S$ to $V(G) \setminus S$. We denote the weight of $I_{A,S}(G)$ by $\alpha_{A,S}(G)$ (i.e. $\alpha_{A, S}(G) = w(I_{A, S}(G))$). The following lemma is the key connection between the two main ingredients of our algorithm: submodular functions and even sets. The proof is similar in spirit to the proof of the statement 6.5 of \cite{CTTV}.

\begin{lemma}
Let $S$ be an even set in a graph $G$. Let $f:2^S \rightarrow \mathbb{R}$ be a set function defined on the subsets of $S$ and given by $f(A) = -\alpha_{A,S} (G)$ for $A \subseteq S$. Then, the function $f$ is submodular.
\label{lemma:submodular}
\end{lemma}

\begin{proof}
Let $A, B$ be two subsets of $S$. Our goal is to show that the following inequality holds:
$$\alpha_{A,S} (G) + \alpha_{B,S} (G) \leq \alpha_{A\cup B,S} (G) + \alpha_{A\cap B,S} (G).$$
For ease of notation, let $I_A = I_{A,S}(G)$ and $I_B = I_{B,S}(G)$. In the bipartite graph $G[I_A \cup I_B]$, we denote by $Y_A$ (resp. $Y_B$) the set of those vertices of $I_A \cup I_B$ for which there exists a path in $G[I_A \cup I_B]$ joining them to a vertex of $A \setminus B$ (resp. $B \setminus A$). Note that by the definition of the sets $Y_A$ and $Y_B$, we have $(A \setminus B) \subseteq Y_A$, $(B \setminus A) \subseteq Y_B$, and $Y_A \cup Y_B$ is anticomplete to $(I_A \cup I_B) \setminus (Y_A \cup Y_B)$. We claim that $Y_A \cap Y_B = \emptyset$ and $Y_A$ is anticomplete to $Y_B$. Suppose not, then there is a path $P$ in $G[I_A \cup I_B]$ from a vertex of $A \setminus B$ to a vertex of $B \setminus A$. We may assume that $P$ is minimal with respect to this property, and so the interior of $P$ is in $V(G) \setminus S$. Hence, $P$ is of odd length because $G[I_A \cup I_B]$ is bipartite. This contradicts the assumption that $S$ is an even set in $G$. Now, we define the following sets:
\begin{equation*}
\begin{split}
S_1 &= (I_A \cap Y_A) \cup (I_B \cap Y_B) \cup (I_A \setminus (Y_A \cup Y_B)),\\
S_2 &= (I_A \cap Y_B) \cup (I_B \cap Y_A) \cup (I_B \setminus (Y_A \cup Y_B)).
\end{split}
\end{equation*}
Since $Y_A \cap Y_B = \emptyset$ and $Y_A$ is anticomplete to $Y_B$, observe that $S_1$ and $S_2$ are independent sets, and that $S_1 \cap S = A \cup B$ and $S_2 \cap S = A \cap B$. Moreover, since $I_A \cup I_B = S_1 \cup S_2$ and $I_A \cap I_B = S_1 \cap S_2$, we have $w(I_A) + w(I_B) = w(S_1) + w(S_2)$. Hence, we obtain
$$\alpha_{A,S} (G) + \alpha_{B,S} (G) = w(I_A) + w(I_B) = w(S_1) + w(S_2) \leq \alpha_{A\cup B,S} (G) + \alpha_{A\cap B,S} (G).$$
This completes the proof.
\end{proof}

Now, we describe how to solve MWIS in polynomial time in graphs with even set separators. The idea of the proof is as follows. Let $G$ be a  $(k, c, d, m)$-tame graph. First, we find a $k$-iterated even set separator $L = (L_1, \hdots, L_k)$ of $G$, which exists since $G$ is tame. Next, we define a sequence of graphs $G_1, \hdots, G_k$, where $G_i = (G \setminus L) \cup \bigcup_{j=0}^{i-1} L_{k-j}$. We iteratively solve MWIS in $G_i$ by computing functions that allow us to find the maximum weight independent set of $G_i$ given a set of vertices that must be included in the independent set. These functions are computed via submodular function minimization, and solutions for $G_{i-1}$ serve as oracles for the submodular function minimization for $G_i$. 
The fact that each of the components of $G \setminus L$ contains at most  $c|V(G)|$ vertices and has at most $d$
attachments in $L$ allows us to use convexity to keep the complexity of the algorithm polynomial-time.

\leqnomode
\evensetseparator*
\begin{proof}
Let $G$ be a $(k,c,d,m)$-tame graph. Note that it follows that every induced subgraph of $G$ is $(k,c,d,m)$-tame. We prove the result by induction on $n=|V(G)|$. Assume inductively that for every proper induced subgraph $H$ of $G$, MWIS can be solved on $H$ in ${\cal O}(|V(H)|^y)$-time. Let $w$ be a weight function on $V(G)$ that is part of the input to MWIS problem on $G$. We now compute the maximum weight (w.r.t. $w$) of an independent set in $G$ (i.e. $\alpha(G)$) in ${\cal O}(n^y)$-time.

\begin{equation}
\longbox{{\it There exists a $k$-iterated even set $L=(L_1,\ldots ,L_k)$ such that for every connected component $D$ of $G\setminus L$,
$|N(D)|\leq d$ and $|D| \leq cn$. Furthermore, such a set can be constructed in ${\cal O}(n^m)$-time.}}
\label{k-iterated-set}\tag{$1.2.1$}
\end{equation}
\\
{\em Proof of \eqref{k-iterated-set}:}
Let $w^* : V(G)\rightarrow [0,1]$ be such that $w^*(v)=\frac{1}{n}$, for every $v\in V(G)$.
Let $L=(L_1, \ldots ,L_k)$ be a $(w^*,k,c,d)$-even set separator that can be constructed in ${\cal O}(n^m)$-time, which exists since $G$ is
$(k,c,d,m)$-tame. 
Let $D$ be a connected component of $G\setminus L$. Then $|N(D)|\leq d$ and $w^*(D)\leq c$.
Since $w^*(D)=|D|\cdot \frac{1}{n}\leq c$, it follows that $|D|\leq cn$.
This proves \eqref{k-iterated-set}. \\

Let $L=(L_1,\ldots ,L_k)$ be the $k$-iterated even set from \eqref{k-iterated-set}, and let $D_1,\ldots ,D_l$ be the connected components of $G\setminus L$. Then for every $i = 1, \ldots ,l$, the following hold: $|N(D_i)|\leq d$, $|D_i|\leq cn$ and for every $i = 1, \hdots, k$, $L_i$ is an even set in $G\setminus (\cup_{j<i} L_j)$.

For $i = 1, \ldots ,l$,  let $g_{D_i}: 2^{N(D_i)} \rightarrow \mathbb{R}$ be a function such that for every $A\subseteq N(D_i)$, it holds that 
$g_{D_i}(A) = \alpha (D_i\setminus N(A))$.

\begin{equation}
\longbox{{\it We can compute the functions $g_{D_1}, \ldots ,g_{D_l}$  in  ${\cal O}(n^y)$-time and store the values in a table $T$.}}\label{table}\tag{$1.2.2$}\end{equation}
{\em Proof of \eqref{table}:}
Since $|D_i|\leq cn$ and $c<1$, it follows from the inductive hypothesis that $g_{D_i} (A)$ can be computed in ${\cal O}(|D_i|^y)$-time, for every $A\subseteq N(D_i)$.
Since $|N(D_i)|\leq d$, there are at most $2^d$ possible sets in the domain of $g_{D_i}$, for all $i=1,\ldots ,l$.
Therefore, functions $g_{D_1}, \ldots ,g_{D_l}$ can be computed in ${\cal O}(2^d\sum_{i=1}^{l} |D_i|^y)$-time. We now show that $2^d \sum_{i=1}^{l} |D_i|^y \leq n^y$. Consider the function $\sum_{i=1}^\ell x_i^y$ (where $y$ is constant) with the constraints $\sum_{i = 1}^\ell x_i \leq n$ and $0 \leq x_i \leq cn$ for all $i = 1, \hdots, \ell$. Since the function $\sum_{i=1}^\ell x_i^y$ is  convex when $x_i \geq 0$ for $i = 1, \hdots, \ell$, and thus over the above feasible set, it follows that it is maximized at an extreme point of the feasible set (see, e.g., \cite[Theorem~7.42]{beck}). The extreme points of the feasible set defined by $0 \leq |D_i| \leq cn$ for $i = 1, \hdots, \ell$ and $\sum_{i = 1}^l |D_i| \leq n$ can easily be enumerated. Indeed, up to a permutation, the extreme points are one of the following types (depending on the relationship between $c$ and $\ell$):
\begin{itemize}[leftmargin=1.3em]
\itemsep0em
\item $(|D_1|, \dots, |D_{\ell}|) = (0, \dots, 0)$,
\item $(|D_1|, \dots, |D_{\ell}|) = (cn, \dots, cn)$,
\item $(|D_1|, \dots, |D_r|, |D_{r+1}|, \dots, |D_{\ell}|) = (cn, \dots, cn, 0, \dots, 0)$ for some $r \in \{1, \hdots, \ell-1 \}$,
\item $(|D_1|, \dots, |D_r|, |D_{r+1}|, |D_{r+2}| \dots, |D_{\ell}|) = (cn, \dots, cn, n-rcn , 0, \dots, 0)$ for some $r \in \{1, \hdots, \ell-1\}$.
\end{itemize}
It is then not difficult to verify that
\begin{itemize}[leftmargin=1.3em]
\itemsep0em
\item if $\ell \leq \frac{1}{c}$, then the sum $\sum_{i=1}^{l} |D_i|^y$ is maximized when $|D_i| = cn$ for $i = 1, \hdots, \ell$, and 
\item if $\ell > \frac{1}{c}$, then the sum $\sum_{i=1}^{l} |D_i|^y$ is maximized when $\lfloor{\frac{1}{c}}\rfloor$ of $|D_i|$'s are equal to $cn$, one of $|D_i|$'s is equal to $n - (cn)\lfloor{\frac{1}{c}}\rfloor$, 
and the rest (if any) are equal to $0$.
\end{itemize}
In the first case $\sum_{i=1}^{l} |D_i|^y \leq \frac{1}{c}(cn)^y$, and in the second case 
$$\sum_{i=1}^{l} |D_i|^y \leq \frac{1}{c}(cn)^y + (n(1 - c\lfloor\frac{1}{c}\rfloor))^y \leq \left(\frac{1}{c} + 1\right) (cn)^y.$$
It follows that $\sum_{i=1}^{l} |D_i|^y \leq \left(\frac{1}{c} + 1\right) (cn)^y$. Since $y \geq z$, we have that $2c^{y-1} \leq \frac{1}{2^{d}}$.
But now since $c \leq 1$, we have $2^d \sum_{i=1}^{l} |D_i|^y \leq n^y$ as required. This proves \eqref{table}. \\

For $i=0, \ldots ,k-1$, let $G_i=L_{k-i}\cup \ldots \cup L_k\cup D_1\cup \ldots D_l$, and let
$f_i:2^{L_{k-i}}\times 2^{L_1\cup \ldots \cup L_{k-i-1}} \rightarrow \mathbb{R}$ be a function such that
for every $B\subseteq L_1\cup \ldots \cup L_{k-i-1}$ and for every $A\subseteq L_{k-i}\setminus N(B)$,
$f_i(A,B)$ is the maximum weight of an independent set $I$ of $G_i\setminus N(B)$ such that $I\cap L_{k-i}=A$.

\begin{equation}
    \longbox{{\it Given a table $T$, and sets $B\subseteq L_1\cup \ldots \cup L_{k-i-1}$ and $A\subseteq L_{k-i}\setminus N(B)$,
$f_i(A,B)$ can be computed  in ${\cal O}(n^{5i+1})$-time.}}\label{computed}\tag{$1.2.3$}\end{equation}
{\em Proof of \eqref{computed}:}
We proceed by induction on $i$.  For $i=0$, $G_0=L_k\cup D_1\cup \ldots \cup D_l$, and for 
$B\subseteq L_1\cup \ldots \cup L_{k-1}$ and $A\subseteq L_{k}\setminus N(B)$,

$$f_0(A,B)=w(A)+\sum_{I=1}^{l} g_{D_i}((A\cup B)\cap N(D_i)).
$$ 
So $f_0(A,B)$ can be computed by $l\leq n$ lookups to the table $T$.

Suppose that given $B\subseteq L_1\cup \ldots \cup L_{k-i-1}$ and $A\subseteq L_{k-i}\setminus N(B)$, $f_i(A,B)$ can be computed in ${\cal O}(n^{5i+1})$-time
using lookups to the table $T$. We now consider $f_{i+1}$.
Let $B\subseteq L_1\cup \ldots \cup L_{k-i-2}$, $G_{i+1}'=G_{i+1}\setminus N(B)$ and $A\subseteq L_{k-i-1}\cap G_{i+1}'$.
We need to show that we can compute $f_{i+1}(A,B)$ in ${\cal O}(n^{5(i+1)+1})$-time.

Note that $f_{i+1}(A,B)$ is the maximum weight of an independent set $I$ of $G_{i+1}'$ such that $I\cap L_{k-i-1}=A$. So 
$$f_{i+1}(A,B) = w(A)+\alpha (G_i'),\text{  where } G_i'=G_i\setminus (N(A)\cup N(B)).$$

By Lemma \ref{lemma:submodular} (applied to the even set $L_{k-i-1}$ and graph $G_i\setminus N(Y)$), the function $-f_i(\cdot ,Y)$ is submodular for every 
$Y\subseteq L_1\cup \ldots \cup L_{k-i-1}$.
So MWIS can be solved in $G_i'$ by minimising $-f_i(\cdot ,A\cup B)$, which by Lemma 2.1 can be computed in ${\cal O}(n^5EO+n^6)$-time,
where EO is the running time of evaluating $f_i(C,A\cup B)$ for a given $C\subseteq L_{k-i}$.
Since by the inductive hypothesis $f_i(C,A\cup B)$  can be computed in ${\cal O}(n^{5i+1})$-time for any $C\subseteq L_{k-i}$,
it follows that $f_{i+1}(A,B)$ can be computed in ${\cal O}(n^{5(i+1)+1})$-time.
This proves \eqref{computed}. \\

Let $T$ be the table constructed in \eqref{table}.
Consider the function $f:2^{L_1} \rightarrow  \mathbb{R}$ such that for every $A\subseteq L_1$, 

\vspace{2ex}

\begin{center}
$f(A)=f_{k-1}(A,\emptyset)=$ maximum weight of an independent set $I$ of $G$ such that $I\cap L_1=A$.
\end{center}

\vspace{2ex}

By Lemma 2.2 (applied to the even set $L_1$ and the graph $G$), $-f$ is submodular, and so by 
Lemma 2.1, $-f$ can be minimized in ${\cal O}(n^5EO+n^6)$-time, where EO is the running time of evaluating $f(A)$.
By \eqref{computed}, $f(A)$ can be evaluated in ${\cal O}(n^{5(k-1)+1})$-time. It follows that $-f$ can be minimized in ${\cal O}(n^{5k+1})$-time.
Therefore, by \eqref{k-iterated-set} and \eqref{table}, MWIS can be computed in $G$ in  ${\cal O}(n^y)$-time.
\end{proof}

\section{Iterated decompositions and their central bags}
\label{sec:central_bags}

Let $G$ be a graph. In what follows, unless otherwise specified, $w: V(G) \to [0, 1]$ is a weight function on $V(G)$ with $w(G) = 1$. A {\em separation of $G$} is a triple $(A, C, B)$ such that the (possibly empty) sets $A$, $B$, $C$ are pairwise vertex-disjoint, $A \cup C \cup B = V(G)$, and $A$ is anticomplete to $B$. When $S = (A, C, B)$ is a separation of $G$, we use the following notation: $A(S) = A$, $C(S) = C$, and $B(S) = B$. A separation $(A, C, B)$ is {\em $d$-bounded} if $C$ is $d$-bounded. For $\varepsilon \in [0, 1]$, a separation $(A, C, B)$ is {$\varepsilon$-skewed} if $w(A) < \varepsilon$ or $w(B) < \varepsilon$.

\begin{lemma}
\label{lemma:skewed}
Let $G$ be a graph. Let $w: V(G) \to [0, 1]$ be a weight function on $V(G)$ with $w(G) = 1$. Let $c \in [\frac{1}{2}, 1)$ and $d$ be a nonnegative integer. Suppose $G$ has no $d$-bounded $(w, c)$-balanced separator. Let $(A, C, B)$ be a $d$-bounded separation of $G$. Then, $(A, C, B)$ is $(1-c)$-skewed.
\end{lemma}

\begin{proof}
Since $G$ has no $d$-bounded $(w, c)$-balanced separator, not both $w(A) \leq c$ and $w(B) \leq c$. We may assume that $w(B) > c$, and so $w(A) < 1 - c$. It follows that $(A, C, B)$ is $(1-c)$-skewed. 
\end{proof}

For the remainder of the paper, if $(A, C, B)$ is $\varepsilon$-skewed, we assume by convention that $w(A) < \varepsilon$. Next, we discuss important relationships between two separations of a graph. Two separations $S_1$ and $S_2$ of a graph $G$ are {\em loosely non-crossing} if $A(S_1) \cap C(S_2) = \emptyset$ and $A(S_2) \cap C(S_1) = \emptyset$. The loosely non-crossing property is similar to the non-crossing property (\cite{RS-GMX}). 
Two separations $S_1$ and $S_2$ are {\em non-crossing} if they are loosely non-crossing and $A(S_1) \cap A(S_2) = \emptyset$. 
Figure \ref{fig:non-crossing} illustrates the properties of non-crossing and loosely non-crossing separations. Note that if two separations are non-crossing, then they are also loosely non-crossing. 
We say that $\S$ is {\em loosely laminar} if $S_1$ and $S_2$ are loosely non-crossing for all $S_1, S_2 \in \S$. 

\begin{figure}[ht]
\begin{subfigure}{0.5\textwidth}   
\centering 
\begin{tabular}{c|c|c|c}
        & $A_1$ & $C_1$ & $B_1$ \\ \hline
    $A_2$    & $\varnothing$ & $\varnothing$ & \\ \hline
    $C_2$ & $\varnothing$ & & \\ \hline
    $B_2$ & & & \\
    \end{tabular}
    \caption{Non-crossing}
\end{subfigure}
\begin{subfigure}{0.5\textwidth}
\centering
        \begin{tabular}{c|c|c|c}
        & $A_1$ & $C_1$ & $B_1$ \\ \hline
    $A_2$    & & $\varnothing$ & \\ \hline
    $C_2$ & $\varnothing$ & & \\ \hline
    $B_2$ & & & \\
    \end{tabular}
    \caption{Loosely non-crossing}
\end{subfigure}
    \caption{Illustrations of two separations $S_1 = (A_1, C_1, B_1)$ and $S_2 = (A_2, C_2, B_2)$ being (a) non-crossing and (b) loosely non-crossing.}
    \label{fig:non-crossing}
\end{figure}

Observe that if $S_1 = (A_1, C_1, B_1)$ and $S_2 = (A_2, C_2, B_2)$ are loosely non-crossing, then every connected component of $A_1 \cup A_2$ is a connected component of $A_1$ or a connected component of $A_2$. If $S_1$ and $S_2$ are $\varepsilon$-skewed, then every connected component of $G[A_1 \cup A_2]$ is ``small'', because $w(A_1), w(A_2) < \varepsilon$. We use this intuition to define a structure called a central bag. From now on, for the remainder of the paper, wherever we refer to a collection of separations $\S$, we will assume that there is a fixed ordering of the separations in $\S$. We then refer to $\S$ as a {\em sequence of separations}. We will explain later how this ordering is obtained. Let $\mathcal{S}$ be a loosely laminar sequence of separations of a graph $G$. The {\em central bag for $\mathcal{S}$}, denoted by $\beta_{\mathcal{S}}$, is defined as follows: 
$$\beta_{\mathcal{S}} = \bigcap_{S \in \mathcal{S}} (B(S) \cup C(S)).$$ 

Note that $V(G) \setminus \beta_\S = \bigcup_{S \in \S} A(S)$. The weight function $w_\S: \beta_\S \to [0, 1]$ is defined as follows. Let $\S = (S_1, \hdots, S_k)$ be a loosely laminar sequence of separations, let $S_i = (A_i, C_i, B_i)$ for every $1 \leq i \leq k$, and for every $S_i$, we assign a vertex $v_i \in$ $C_i$ to be the {\em anchor} of $S_i$. Any vertex of $C_i$ can be the anchor for $S_i$, but our choice of anchor is important and depends on what type of separation $S_i$ is. When we discuss separations in ($C_4$, prism)-free perfect graphs, we will describe how to choose the anchors wisely. In the next lemma, we show that no matter how we choose anchors, $\{v_1, \hdots, v_k\} \subseteq \beta_\S$. We will only consider situations in which every vertex of $G$ is the anchor of at most one separation in $\S$. For $v \in \beta_\S$, $w_\S(v) = w(v) + w(A_i \setminus \bigcup_{1 \leq j < i} A_j)$ if $v = v_i$ for some $i \in \{1, \hdots, k\}$, and $w_\S(v) = w(v)$ otherwise. (We note that this is well-defined by the assumption that every vertex of $G$ is the anchor of at most one separation in $\mathcal{S}$). 
Thus, the anchor of a separation $S$ is a way to record the weight of $A(S)$ in $\beta_\S$. The following lemma gives important properties of central bags. 

\begin{lemma}
Let $\varepsilon > 0$ and let $G$ be a graph. Let $w: V(G) \to [0, 1]$ be a weight function on $V(G)$ with $w(G) = 1$. Let $\mathcal{S}$ be a loosely laminar sequence of separations of $G$ such that $S$ is $\varepsilon$-skewed for all $S \in \S$, and every vertex $v \in V(G)$ is the anchor of at most one separation in $\S$. Let $\beta_\S$ be the central bag for $\S$. Then,

\begin{enumerate}[(i)]
\itemsep -0.2em
\item $C(S) \subseteq \beta_\S$ for every $S \in \S$,
\item if $G$ is connected and $C(S)$ is connected for every $S \in \S$, then $\beta_\S$ is connected,
\item $w_\S(\beta_\S) = 1$ and $w_\S^{\max} \leq w^{\max} + \varepsilon$.
\end{enumerate}
\label{lemma:central_bag_1}
\end{lemma}
\begin{proof}
Let $S_i = (A_i, C_i, B_i)$ for all $1 \leq i \leq k$. Since $\S$ is a loosely laminar sequence of separations, we have $C_i \cap A_j = \emptyset$ for all $i, j \in \{1, \hdots, k\}$. Since $G \setminus \beta_\S \subseteq \bigcup_{1 \leq i \leq k} A_i$, it follows that $C_i \cap (G \setminus \beta_\S) = \emptyset$ for all $1 \leq i \leq k$. Therefore, $C_i \subseteq \beta_\S$ for all $1 \leq i \leq k$. This proves (i). 

To prove (ii), suppose that $G$ is connected and $C(S)$ is connected for every $S \in \S$. Let $D$ be a connected component of $\beta_\S$. Let $I = \{i : C_i \cap D \neq \emptyset\}$. Since $C_i \subseteq \beta_\S$ and $C_i$ is connected, it follows that $C_i \subseteq D$ for all $i \in I$. Since $N(A_i) \subseteq C_i$, we deduce that $D \cup \bigcup_{i \in I} A_i$ is a connected component of $G$. Since $G$ is connected, it follows that $D \cup \bigcup_{i \in I} A_i = V(G)$, and so $D = \beta_\S$. This proves (ii). 

For $1 \leq i \leq k$, let $v_i$ be the anchor for $S_i$. Note that
\begin{align*}
\hspace{3cm}
w_\S(\beta_\S)
&= \sum_{v \in \beta_\S \setminus \{v_1, \hdots, v_k\}} w_\S(v) + \sum_{v \in \{v_1, \hdots, v_k\}} w_\S(v) \\
&= \sum_{v \in \beta_\S} w(v) + \sum_{1 \leq i \leq k} w\left(A_i \setminus \bigcup_{1 \leq j < i} A_j\right) \\
&= \sum_{v \in V(G)} w(v),
\end{align*}
where the last equality holds since $V(G) \setminus \beta_\S = \bigcup_{S \in \S} A(S)$. Since $w(G) = 1$, it follows that $w_\S(\beta_\S) = 1$. For every $v \in \beta_\S$, $w_\S(v) \leq w(v) + \max_{1 \leq i \leq k} w(A_i)$. Since $S_i$ is $\varepsilon$-skewed, it follows that $w(A_i) \leq \varepsilon$ for all $1 \leq i \leq k$. Therefore, $w_\S^{\max} \leq w^{\max} + \varepsilon$. This proves (iii). 
\end{proof}

Let $\S$ be a loosely laminar sequence of separations of $G$ and let $\beta_\S$ be the central bag for $\S$. Let $S_1 = (A_1, C_1, B_1)$ and $S_2 = (A_2, C_2, B_2)$ be two separations in $\S$. Suppose $C_1 \cup B_1 \subseteq C_2 \cup B_2$. Since $\beta_\S = \bigcap_{S \in \S} (B(S) \cup C(S))$, it follows that $\beta_{\S \setminus S_2} = \beta_{\S}$. Therefore, $S_2$ is in some sense an ``unnecessary'' member of $\S$. We formalize this notion by defining shields. If $S_1$ and $S_2$ are separations, and $C(S_1) \cup B(S_1) \subseteq C(S_2) \cup B(S_2)$, then $S_1$ is called a {\em shield for $S_2$}. (Equivalently, $S_1$ is a shield for $S_2$ if $A(S_2) \subseteq A(S_1))$.) Later, we will use shields to define a notion of a ``minimal'' sequence of separations.

\section{Star separations}
\label{sec:star_separations}
Let $G$ be a graph and let $w: V(G) \to [0, 1]$ be a weight function on $V(G)$ with $w(G) = 1$. A separation $S = (A, C, B)$ is a {\em star separation} if there exists $v \in C$ such that $C \subseteq N[v]$.

Let $v \in V(G)$. A {\em canonical star separation for $v$}, denoted $S_v = (A_v, C_v, B_v)$, is defined as follows: $B_v$ is a largest weight connected component of $G \setminus N[v]$, $C_v$ is the union of $v$ and every vertex in $N(v)$ with a neighbor in $B_v$, and $A_v = V(G) \setminus (B_v \cup C_v)$. Observe that if $A_v$ is non-empty, then $C_v$ is a star cutset of $G$. We call $v$ the {\em center} of $S_v$. When we decompose using canonical star separations $S_v$, the anchor of $S_v$ is its center $v$. 

Every result in this section has the following common assumptions: \\

\noindent \textbf{Common assumptions for Section 4:} {\em Let $c \in [\frac{1}{2}, 1)$ and let $d, \delta$ be positive integers with $d \geq \delta$. Let $G$ be a graph and $w:V(G) \to [0, 1]$ a weight function on $V(G)$ with $w(G) = 1$. Assume that $G$ has maximum degree $\delta$ and no $d$-bounded
$(w, c)$-balanced separator.} \\

The following lemma shows that in graphs with no bounded balanced separator, canonical star separations are unique. 

\begin{lemma}
\label{lemma:canonical_unique}
For every $v \in V(G)$, the canonical star separation for $v$ is unique.
\end{lemma}
\begin{proof}
Let $v \in V(G)$. Since $N[v]$ is $d$-bounded and $G$ has no $d$-bounded $(w, c)$-balanced separator, it follows that there exists a connected component $B$ of $G \setminus N[v]$ such that $w(B) > c$. Since $c \geq \frac{1}{2}$, $B$ is the unique largest weight connected component of $G \setminus N[v]$. Therefore, $B = B_v$. Now, since a canonical star separation for $v$ is uniquely defined by $B_v$, the result follows.
\end{proof}

We say that two vertices $u, v \in V(G)$ are {\em star twins} if $B_u = B_v$, $C_u \setminus \{u\} = C_v \setminus \{v\}$, and $A_u \setminus \{v\} = A_v \setminus \{u\}$ (see Figure \ref{fig:star_twins}). Note that for star twins $u, v$, we have $u \in A_v$ and $v \in A_u$. If $u$ and $v$ are star twins, we also say that their canonical star separations $S_u$ and $S_v$ are {\em star twins}. Recall that $S_v$ is a {\em shield} for $S_u$ if $B_v \cup C_v \subseteq B_u \cup C_u$, and so if $u$ and $v$ are star twins, then $S_v$ is not a shield for $S_u$ and $S_u$ is not a shield for $S_v$. The following lemma characterizes shields of star separations.

\begin{figure}[h]
\begin{center}
\begin{tikzpicture}[scale=0.6]

\node[label=above:{\scriptsize{$u$}}, inner sep=2pt, fill=black, circle] at (-6.5, 0.5)(u){};
\node[label=above:{\scriptsize{$v$}}, inner sep=2pt, fill=black, circle] at (-3, 0.5)(v){};

\tikzset{decoration={snake,amplitude=.4mm,segment length=2mm,
                       post length=0mm,pre length=0mm}}
\draw[decorate] (u) -- (v);

\draw (v) -- (-3.5,-1.2);
\draw (v) -- (-3,-1.2);
\draw (v) -- (-2.5,-1.2);

\draw (u) -- (-3.7,-1.4);
\draw (u) -- (-3.1,-1.4);
\draw (u) -- (-2.5,-1.4);

\draw (-6,-1.5) -- (-6,1.5) arc(90:270:1.5) --cycle;

\draw (0,1.5) -- (0,-1.5) arc(-90:90:1.5) --cycle;

\draw (-1.5,-1) -- (-4.5,-1) arc(180:360:1.5) --cycle;

\coordinate[label={\scriptsize{$A_v$}}] (E) at (-6.5,-2.5);
\coordinate[label={\scriptsize{$C_v \setminus \{v\} = C_u \setminus \{u\}$}}] (F) at (-3,-3.5);
\coordinate[label={\scriptsize{$B_v = B_u$}}] (F) at (0.75,-2.5);

\end{tikzpicture}
\end{center}
\vspace{-0.6cm}
\caption{Star twins $u$ and $v$, where $u$ and $v$ may or may not be adjacent}
\label{fig:star_twins}
\end{figure}

\begin{lemma}
\label{lemma:star_shields}
Let $u, v \in V(G)$, and suppose $u \in A_v$. Then, either $S_u$ and $S_v$ are star twins, or $S_v$ is a shield for $S_u$.
\end{lemma}
\begin{proof}
Since $u$ is anticomplete to $B_v$, we have $B_v \subseteq G \setminus N[u]$. Since $N[v]$ is $d$-bounded and $G$ has no $d$-bounded $(w, c)$-balanced separator, it follows that $w(B_v) > c$, so $B_v \subseteq B_u$. Let $x \in C_v \setminus \{v\}$, so $x$ has a neighbor in $B_v$ and thus in $B_u$. If $x \in N[u]$, then $x \in C_u$. If $x \not \in N[u]$, then $x \in B_u$. It follows that $C_v \setminus \{v\} \subseteq C_u \cup B_u$. 

Suppose there exists $y \in C_v \setminus \{v\}$ such that $y \in B_u$. Then, since $v$ is adjacent to $y$, we have $v \in B_u \cup C_u$, and so $S_v$ is a shield for $S_u$. Therefore, we may assume $(C_v \setminus \{v\}) \cap B_u = \emptyset$, so $C_v \setminus \{v\} \subseteq C_u$. Since $C_v \setminus \{v\} \subseteq C_u$ and by definition $N(B_v) = C_v \setminus \{v\}$, it follows that $B_v = B_u$ and $C_v \setminus \{v\} = C_u \setminus \{u\}$. Then, $A_v \setminus \{u\} = A_u \setminus \{v\}$. Therefore, $u$ and $v$ are star twins. 
\end{proof}

Lemma \ref{lemma:star_shields} allows us to define a useful relation on $V(G)$. For the remainder of the paper we fix an ordering $\smallO$ of $V(G)$. Let $\leq_A$ be a relation on $V(G)$ defined as follows: 
\begin{equation*}
\hspace{2.5cm}
x \leq_A y \ \ \ \text{ if} \ \ \  
\begin{cases} x = y, \text{ or} \\ 
\text{$x$ and $y$ are star twins and $\smallO(x) < \smallO(y)$, or}\\ 
\text{$x$ and $y$ are not star twins and } y \in A_x.\\
\end{cases}
\end{equation*} 

\begin{lemma}
\label{lemma:partial-order}
The ordering $\leq_A$ is a partial order on $V(G)$. 
\end{lemma}
\begin{proof}
We show that $\leq_A$ is reflexive, antisymmetric, and transitive. By definition, $\leq_A$ is reflexive. Suppose $x \leq_A y$ and $y \leq_A x$ for some $x, y \in V(G)$ with $x \neq y$. Since $x \leq_A y$, it follows that $y \in A_x$, and since $y \leq_A x$, it follows that $x \in A_y$. Hence, $A_x \not\subseteq A_y$ and $A_y \not\subseteq A_x$, and so $S_x$ and $S_y$ are not shields for each other. Thus, by Lemma \ref{lemma:star_shields}, $x$ and $y$ are star twins. But $x \leq_A y$ implies that $\smallO(x) < \smallO(y)$, and $y \leq_A x$ implies that $\smallO(y) < \smallO(x)$, a contradiction. Therefore, $\leq_A$ is antisymmetric. 

Suppose $x \leq_A y$ and $y \leq_A z$ for $x, y, z \in V(G)$ distinct. Since $y \leq_A z$ and $y \neq z$, it follows that $z \in A_y$ (note that $z \in A_y$ if $y$ and $z$ are star twins). Similarly, since $x \leq_A y$ and $x \neq y$, it follows that $y \in A_x$, so by Lemma \ref{lemma:star_shields}, $A_y \setminus \{x\} \subseteq A_x$. Therefore, $z \in A_x$. If $x$ and $z$ are not star twins, then $x \leq_A z$, so assume $x$ and $z$ are star twins. 

Since $x$ and $z$ are star twins, $A_x \setminus \{z\} = A_z \setminus \{x\}$. Since $y \in A_x$, it follows that $y \in A_z$. Since $A_y \not\subseteq A_z$, by Lemma \ref{lemma:star_shields}, $y$ and $z$ are star twins and $\smallO(y) < \smallO(z)$. Since $y$ and $z$ are star twins, $A_y \setminus \{z\} = A_z \setminus \{y\}$. Since $x \in A_z$, it follows that $x \in A_y$, so $x$ and $y$ are star twins and $\smallO(x) < \smallO(y)$. Therefore, $\smallO(x) < \smallO(z)$, and $x \leq_A z$, so $\leq_A$ is transitive.
\end{proof}

Let $X \subseteq V(G)$ be the set of minimal vertices with respect to $\leq_A$, and let $\S = \{S_x : x \in X\}$ be the set of canonical star separations with centers in $X$. We call $X$ the {\em $\smallO$-star covering} of $G$, and the ordering of $\S$ by the order of the centers of the separations in $\S$ with respect to $\smallO$ the {\em $\smallO$-star covering sequence} of $G$. In fact, we will assume that every collection of canonical star separations is ordered by the order of its centers with respect to $\smallO$.  The {\em dimension of $\S$}, denoted $\dim(\S)$, is the minimum $k$ such that $\S$ can be partitioned into $k$ loosely laminar sequences. The following two lemmas show that $\dim(\S)$ is bounded above by the chromatic number of $X$. 

\begin{lemma}
\label{lemma:loosely_noncrossing}
Let $x, y \in V(G)$ be such that $x$ and $y$ are non-adjacent and incomparable with respect to $\leq_A$. Then, $S_x$ and $S_y$ are loosely non-crossing.
\end{lemma}
\begin{proof}
Since $x$ and $y$ are incomparable with respect to $\leq_A$, it follows that $x \not \in A_y$ and $y \not \in A_x$. Since $x$ and $y$ are non-adjacent, it follows that $x \in B_y$ and $y \in B_x$. Since $A_y$ is anticomplete to $B_y$, it follows that $C_x \subseteq B_y \cup C_y$, so $C_x \cap A_y = \emptyset$. By symmetry, $C_y \cap A_x = \emptyset$. Therefore, $S_x$ and $S_y$ are loosely non-crossing.
\end{proof}

\begin{lemma}
Let $X$ be the $\smallO$-star covering of $G$, and let $\S$ be the $\smallO$-star covering sequence of $G$. Then, $\dim(\S) \leq \chi(X)$.
\label{lemma:dim_equals_chi}
\end{lemma}
\begin{proof}
Let $\chi(X) = k$, let $X_1, \hdots, X_{k}$ be a partition of $X$ into independent sets, and for all $1 \leq i \leq k$, let $\S_i$ be the set of separations of $\S$ with centers in $X_i$. Let $u, v \in X_i$. Since $X$ is an antichain of $\leq_A$, it follows that $u$ and $v$ are incomparable with respect to $\leq_A$. By Lemma \ref{lemma:loosely_noncrossing}, $S_u$ and $S_v$ are loosely non-crossing. It follows that $\S_i$ is loosely laminar. Therefore, $\S_1, \hdots, \S_{k}$ is a partition of $\S$ into $k$ loosely laminar sequences, so $\dim(\S) \leq k$. 
\end{proof}

Lemma \ref{lemma:dim_equals_chi} shows that the dimension of the star covering sequence of a graph is bounded by its chromatic number. Therefore, the following is an immediate corollary of Lemma \ref{lemma:dim_equals_chi}.

\begin{lemma}
Let $\S$ be the $\smallO$-star covering sequence of $G$. Then, $\dim(S) \leq \delta + 1$. 
\label{lemma:dimS}
\end{lemma}

Now, we extend the idea of central bag from a single loosely laminar sequence of separations to the star covering sequence $\S$. Let $X$ be the $\smallO$-star covering of $G$ and let $\S$ be the $\smallO$-star covering sequence for $G$. The {\em star-free bag of $G$}, denoted $\beta$, is defined as follows:
$$ \beta = \bigcap_{S \in \S} (B(S) \cup C(S)).$$

We will later show that $\beta = X$ under some conditions. Note that if $\S$ is loosely laminar, then $\beta = \beta_\S$. The following lemma states an essential result about central bags of sequences of loosely non-crossing star separations: they have no bounded balanced separator. 
\leqnomode
\begin{lemma}
\label{lemma:central_bag_2}
Let $\S=(S_1, \ldots, S_s)$ be a loosely laminar sequence of star separations such that every $v \in V(G)$ is the center of at most one separation in $\S$, and let $\beta_\S$ be the central bag for $\S$. Then $\beta_\S$ has no $(d-1)$-bounded $(w_\S, c)$-balanced separator.
\end{lemma}
\begin{proof}
Let $\S = (S_1, \hdots, S_s)$. Suppose $\beta_\S$ has a $(d-1)$-bounded $(w_\S, c)$-balanced separator $Y$.  Since $G$ has no $d$-bounded $(w, c)$-balanced separator, it follows that $N[Y] \cap \beta_\S$ is not a $d$-bounded $(w, c)$-balanced separator of $G$. Since  $N[Y] \cap \beta_\S$ is $d$-bounded, it follows that there exists a connected component $X$ of $G \setminus (N[Y] \cap \beta_\S)$ such that $w(X)>c$. Let $Q_1, \hdots, Q_t$ be the connected components of $\beta_\S \setminus Y$ and let  $D_1, \hdots, D_m$ be the connected components of $G \setminus \beta_\S$. Define $\mathcal{I}=\{i \mid Q_i \cap X \neq \emptyset\}$ and $\mathcal{J}=\{j \mid D_j \cap X \neq \emptyset\}$. Recall (from Section \ref{sec:central_bags}) that $V(G) \setminus \beta_\S = \bigcup_{S \in \S} A(S)$ and that for $S, S' \in \S$, every connected component of $A(S) \cup A(S')$ is either a connected component of $A(S)$ or a connected component of $A(S')$. It follows that for every $1 \leq j \leq m$, there exists $1 \leq k \leq s$ such that $D_j \subseteq A(S_k)$. 
  For $j \in \mathcal{J}$ let $f(j)$ be minimum such that
  $D_j \subseteq A(S_{f(j)})$. Let $S(j)=S_{f(j)}$ and let
  $v(j)$ be the center of $S(j)$.
\begin{equation} 
\mbox{{\it For every $j \in \mathcal{J}$ we have that $D_j \subseteq X$, $C(S(j)) \not \subseteq N[Y] \cap \beta_\S$, and $v(j) \not \in Y$.}}
\label{sepcomp}
\tag{4.7.1}
\end{equation}
{\em Proof of \eqref{sepcomp}}: Since $D_j \cap X \neq \emptyset$, and  $X$ is a connected component of $G \setminus N[Y] \cap \beta_\S$ and $D_j$ is a connected component of $G \setminus \beta_\S$, it follows that $D_j \subseteq X$. Suppose that $C(S(j)) \subseteq N[Y] \cap \beta_\S$. Since  $N(D_j) \subseteq C(S(j)) \subseteq  N[Y] \cap \beta_\S$, and since $X$ is connected, it follows that $X=D_j$. Now, $w(X) = w(D_j) \leq 1-c$ (by Lemma \ref{lemma:skewed} and since $D_j \subseteq A(S_{f(j)})$), a contradiction since $w(X) > c$ and $c \geq \frac{1}{2}$. This proves the second claim of \eqref{sepcomp}. Next, we observe that if $v(j) \in Y$, then $C(S(j)) \subseteq N[Y] \cap \beta_\S$, and thus the third claim of \eqref{sepcomp} follows from the second. This proves \eqref{sepcomp}.\\

Suppose first that $\mathcal{I}=\emptyset$. Then $\mathcal{J} \neq \emptyset$. Let $j \in \mathcal{J}$. Since $(N[Y]\cap \beta_\S) \cup Q_1 \cup \hdots \cup Q_t = \beta_\S$, it follows that $X = D_j$ and so $C(S(j)) \subseteq N[Y] \cap \beta_\S$, contrary to \eqref{sepcomp}.

Next suppose that $|\mathcal{I}| \geq 2$. Then there exist
$i, j \in \mathcal{I}$ and $k \in \mathcal{J}$ such that 
$N(D_k) \cap Q_i \neq \emptyset$ and $N(D_k) \cap Q_j \neq \emptyset$.
Then, $C(S(k)) \cap Q_i \neq \emptyset$ and $C(S(k)) \cap Q_j \neq \emptyset$. Since $v(k)$ is complete to $C(S(k)) \setminus \{v(k)\}$ and $Q_a$ and $Q_b$ are anticomplete to each other for all $1 \leq a, b \leq t$, it follows that $v(k) \not \in Q_\ell$ for any $1 \leq \ell \leq t$. By (i) of Lemma \ref{lemma:central_bag_1}, $v(k) \in \beta_\S$, so it follows that $v(k) \in Y$, contrary to
\eqref{sepcomp}.

We have shown  that $|\mathcal{I}|=1$,
say $\mathcal{I}=\{i\}$. Now let $j \in \mathcal{J}$.
Since $X \cap D_j \neq \emptyset$ and $X$ is connected, it follows that
$C(S(j)) \cap Q_i \neq \emptyset$. Therefore, $v(j) \in Q_i$ or $v(j) \in Y$.
By \eqref{sepcomp} we may assume that $v(j) \in Q_i$.

Note that by definition of $S(j)$, $D_j \cap A(S_k) = \emptyset$ for every $k < f(j)$. So $w_\S(v(j)) = w(v(j)) + w(A(S(j)) \setminus \bigcup_{1 \leq k < f(j)} A(S_k)) \geq w(v(j)) + w(D_j)$. Then, 
$$w_\S(Q_i)  \geq  w(Q_i) + \Sigma_{j \in \mathcal{J} \text {  s.t. } v(j) \in Q_i} w(D_j) \geq w(X)$$
and  we have $w(X) \leq w_\S(Q_i) \leq c$, again a contradiction (where the second inequality holds since $Y$ is a $(w_\S, c)$-balanced separator of $\beta_\S$).
\end{proof}

Together, Lemmas \ref{lemma:central_bag_1} and \ref{lemma:central_bag_2} give four key properties for central bags of sequences of loosely laminar star separations. In the next lemma, we extend the results of Lemmas \ref{lemma:central_bag_1} and \ref{lemma:central_bag_2} to sequences of separations of bounded dimension.

\begin{lemma}
\label{lemma:star_free_bag}
Let $X$ be the $\smallO$-star covering of $G$, let $\S$ be the $\smallO$-star covering sequence of $G$, and suppose $\dim(\S) = k$. Let $\beta$ be star-free bag of $G$. Suppose $G$ is connected and $d \geq k$. Then, the following hold:
\begin{enumerate}[(i)]
\itemsep -0.2em
    \item $X \subseteq \beta$,
    
    \item $\beta$ is connected,
   
    \item there exists a weight function $w_\beta: \beta \to [0, 1]$ with $w_\beta(\beta) = 1$ and $w_\beta^{\max} \leq w^{\max} + (1-c)$,
    
    \item $\beta$ has no $(d-k)$-bounded $(w_\beta, c)$-balanced separator. 
\end{enumerate}
\end{lemma}
\begin{proof}
Let $u \in X$. Since $X$ is an antichain of $\leq_A$, it follows that $u \not \in A_v$ for all $v \in X$. Therefore, $u \in B_v \cup C_v$ for every $v \in X$, so $u \in B(S) \cup C(S)$ for every $S \in \S$. It follows that $ u \in \beta$. This proves (i).

Let $\S_1, \hdots, \S_k$ be a partition of $\S$ into loosely laminar sequences (whose orderings are inherited from $\S$), and let $X_1, \hdots, X_k$ be the partition of $X$ into the centers of $\S_1, \hdots, \S_k$. To prove (ii), (iii), and (iv), we will prove inductively that there exists a sequence $\beta_1, \hdots, \beta_k$ with $\beta_i = \bigcap_{S \in \S_1 \cup \hdots \cup \S_i}(B(S) \cup C(S))$, such that $\beta_i$ is connected, there exists a weight function $w_i: \beta_i \to [0, 1]$ such that $w_i(\beta_i) = 1$, $w_i(v) = w(v)$ if $v \not \in X_1 \cup \hdots \cup X_i$, and $w_i(v) \leq w^{\max} + (1-c)$ if $v \in X_1 \cup \hdots \cup X_i$, and $\beta_i$ has no $(d-i)$-bounded $(w_i, c)$-balanced separator. 

Let $S$ be a separation of $G$ and let $H$ be an induced subgraph of $G$. Then, $S \cap H$ is the separation given by $(A(S) \cap H, C(S) \cap H, B(S) \cap H)$. Let $\beta_1$ be the central bag for $\S_1$, so $\beta_1 = \bigcap_{S \in \S_1}(B(S) \cup C(S))$. Let $w_1 = w_{\S_1}$ be the weight function on $\beta_1$. By Lemma \ref{lemma:canonical_unique}, every vertex of $G$ is the anchor of at most one separation in $\S_1$, and by Lemma \ref{lemma:skewed}, every separation of $\S_1$ is $(1-c)$-skewed. Then, by Lemma \ref{lemma:central_bag_1}, $\beta_1$ is connected and $w_1(\beta_1) = 1$, and by Lemma \ref{lemma:central_bag_2}, $\beta_1$ has no $(d-1)$-bounded $(w_1, c)$-balanced separator. Further, $w_1(v) = w(v)$ if $v \not \in X_1$, and $w_1(v) \leq w^{\max} + (1-c)$ if $v \in X_1$. 

Now, for $i < k$, suppose $\beta_i = \bigcap_{S \in \S_1 \cup \hdots \cup \S_i} (B(S) \cup C(S))$, $\beta_i$ is connected, $w_i: \beta_i \to [0, 1]$ is a weight function on $\beta_i$ with $w_i(\beta_i) = 1$, and $\beta_i$ has no $(d-i)$-bounded $(w_i, c)$-balanced separator. Further, suppose $w_i(v) = w(v)$ if $v \not \in X_1 \cup \hdots \cup X_i$, and $w_i(v) \leq w^{\max} + (1-c)$ if $v \in X_1 \cup \hdots \cup X_i$. Let $\S_{i+1} \mid \beta_i = \{S \cap \beta_i \mid S \in \S_{i+1}\}$. Note that by (i), all separations in $\S_{i+1} \mid \beta_i$ are star separations and we keep the same anchors for them as for the separations in $\S_{i+1}$. By Lemma \ref{lemma:canonical_unique}, every vertex of $G$ is the anchor of at most one separation in $\S_{i+1}$, and hence every vertex of $\beta_i$ is the anchor of at most one separation in $\S_{i+1}\mid \beta_i$. It follows that $\S_{i+1} \mid \beta_i$ is a loosely laminar sequence of star separations of $\beta_i$. Let $\beta_{i+1}$ be the central bag for $\S_{i+1} \mid \beta_i$, and let $w_{i+1}$ be the weight function for $\beta_{i+1}$. Then,
\begin{align*}
\hspace{4.5cm}
\beta_{i+1} &= \bigcap_{S \in \S_{i+1}\mid \beta_i} (B(S) \cup C(S)) \\
&= \bigcap_{S \in \S_{i+1}} (B(S) \cup C(S) \cap \beta_i) \\
&= \bigcap_{S \in \S_1 \cup \hdots \cup \S_{i+1}} (B(S) \cup C(S))
\end{align*}
Since $d - i \geq 1$ and $\beta_i$ has no $(d-i)$-bounded $(w_i, c)$-balanced separator, by Lemma \ref{lemma:skewed}, every separation in $\S_{i+1} \mid \beta_i$ is $(1-c)$-skewed, and so by Lemma \ref{lemma:central_bag_1}, $\beta_{i+1}$ is connected and $w_{i+1}(\beta_{i+1}) = 1$. Further, $w_{i+1}(v) = w_i(v)$ if $v \not \in X_{i+1}$, and $w_{i+1}(v) \leq w_i(v) + (1-c)$ if $v \in X_{i+1}$. It follows that $w_{i+1}(v) = w(v)$ if $v \not \in X_1 \cup \hdots \cup X_{i+1}$, and $w_{i+1}(v) \leq w^{\max} + (1-c)$ if $v \in X_1 \cup \hdots \cup X_{i+1}$. By Lemma \ref{lemma:central_bag_2}, since $\beta_i$ does not have a $(d-i)$-bounded $(w_i, c)$-balanced separator, it follows that $\beta_{i+1}$ does not have a $(d-(i+1))$-bounded $(w_{i+1}, c)$-balanced separator. 

Let $w_\beta = w_k$. Then, $\beta$ is connected, $w_\beta(\beta) = 1$ and $w_\beta^{\max} \leq w^{\max} + (1-c)$, and $\beta$ has no $(d-k)$-bounded $(w_\beta, c)$-balanced separator. This proves (ii), (iii), and (iv).
\end{proof}

Next, we show that the star-free bag and the star covering set are equivalent. 

\begin{lemma}
\label{lemma:beta_is_X}
Let $X$ be the $\smallO$-star covering of $G$, let $\S_X$ be the $\smallO$-star covering sequence of $G$, and suppose $d \geq \dim(\S_X)$. Let $\beta$ be the star-free bag of $G$. Assume $G$ is connected. Then, $\beta = X$.  
\end{lemma}
\begin{proof}
By Lemma \ref{lemma:star_free_bag}, $X \subseteq \beta$. Let $v \in \beta$ and suppose $v \not \in X$. Then, there exists $x \in X$ such that $x \leq_A v$, so $v \in A_x$. But $\beta \subseteq B_x \cup C_x$, a contradiction. Therefore, $\beta \subseteq X$. 
\end{proof}

The star-free bag $\beta$ of $G$ has a crucial relationship to the star separations of $G$. This relationship is formalized through the idea of forcers. Let $G$ be a graph, $X \subseteq V(G)$ and $v \in V(G) \setminus X$. We say that $v$ {\em breaks} $X$ if for every connected component $D$ of $G \setminus N[v]$, we have $X \not \subseteq N[D]$. We recall that if $G'$ and $G$ are graphs, we say $H$ is a $G'$-copy in $G$ if $H \subseteq V(G)$ and $H$ is isomorphic to $G'$. A graph $F$ is a \emph{forcer for $G$} if for every $Y \subseteq V(G)$ such that $Y$ is an $F$-copy in $G$, there exists a vertex $v \in Y$ such that $v$ breaks $Y \setminus \{v\}$. Such a vertex $v$ is called an \emph{$F$-center} for $Y$. If $F$ is a forcer for every graph in a hereditary class $\mathcal{C}$, we say that $F$ is a {\em forcer for $\mathcal{C}$}. The following lemma shows how forcers relate to the star-free bag.

\begin{lemma}
\label{lemma:bag_no_forcers}
Let $\beta$ be the star-free bag of $G$.
Let $F$ be a forcer for $G$. Then $\beta$ is $F$-free.
\end{lemma}
\begin{proof}
Let $X$ be the $\smallO$-star covering of $G$ and let $\S$ be the $\smallO$-star covering sequence of $G$. Let $Y$ be an $F$-copy in $G$. We will prove that $Y \not \subseteq \beta$. Let $v$ be an $F$-center for $Y$. Suppose $v \in X$. Then, $\beta \subseteq B_v \cup C_v$. Since $v$ is an $F$-center for $Y$, it follows that $v$ breaks $Y \setminus \{v\}$, and so $A_v \cap Y \neq \emptyset$, and thus $Y \not \subseteq \beta$. Therefore, we may assume that $v \not \in X$. Since $X$ is the set of minimal vertices with respect to the relation $\leq_A$, it follows that there exists $u \in X$ such that $u \leq_A v$. Then, $v \in A_u$. Since $\beta \subseteq B_u \cup C_u$ and $v \in Y$, it follows that $Y \not \subseteq \beta$. 
\end{proof}

Because of Lemma \ref{lemma:bag_no_forcers}, forcers for $G$ restrict the structure of the star-free bag $\beta$. We see the power of forcers in the next section, where we discuss forcers for ($C_4$, prism)-free perfect graphs.

\newpage
\section{Forcers for $(C_4$, prism)-free perfect graphs}
\label{sec:paws_are_forcers}

A {\em paw} is a graph with vertex set $\{v_1, v_2, v_3, v_4\}$ and edge set $\{v_1v_2, v_2v_3, v_3v_4, v_2v_4\}$ (see Figure \ref{fig:paw_graph}). In this section, we prove that paws are forcers for the class of ($C_4$, prism)-free perfect graphs. 

\begin{figure}[h]
\begin{center}
\begin{tikzpicture}[scale=0.5]

\node[label=left:{\small{$v_1$}}, inner sep=2pt, fill=black, circle] at (0, 2)(v1){};
\node[label=left:{\small{$v_2$}}, inner sep=2pt, fill=black, circle] at (0, 0)(v2){};
\node[label=left:{\small{$v_3$}}, inner sep=2pt, fill=black, circle] at (-1, -2)(v3){};
\node[label=right:{\small{$v_4$}}, inner sep=2pt, fill=black, circle] at (1, -2)(v4){};

\draw[black, thick] (v1) -- (v2);
\draw[black, thick] (v2) -- (v3);
\draw[black, thick] (v2) -- (v4);
\draw[black, thick] (v3) -- (v4);

\end{tikzpicture}
\end{center}
\vspace{-0.4cm}
\caption{The paw graph}
\label{fig:paw_graph}
\end{figure}

We start with the following lemma (see Lemma 18 in \cite{ISK4} for a variation of this lemma). This proof of Lemma \ref{lem:three_leaves} originally appeared in \cite{wallpaper}, but we include it here for completeness. 

\begin{lemma}
Let $x_1, x_2, x_3$ be three distinct vertices of a graph $G$. Assume that $H$ is a connected induced subgraph of $G \setminus \{x_1, x_2, x_3\}$ such that $V(H)$ contains at least one neighbor of each of $x_1$, $x_2$, $x_3$, and that $V(H)$ is minimal subject to inclusion. Then, one of the following holds:
\begin{enumerate}[(i)]
\item For some distinct $i,j,k \in  \{1,2,3\}$, there exist $P$ that is either a path from $x_i$ to $x_j$ or a hole containing the edge $x_ix_j$ such that
\begin{itemize}
\item $V(H) = V(P) \setminus \{x_i,x_j\}$, and
\item either $x_k$ has at least two non-adjacent neighbors in $H$ or $x_k$ has exactly two neighbors in $H$ and its neighbors in $H$ are adjacent.
\end{itemize}

\item There exists a vertex $a \in V(H)$ and three paths $P_1, P_2, P_3$, where $P_i$ is from $a$ to $x_i$, such that 
\begin{itemize}
\item $V(H) = (V(P_1) \cup V(P_2) \cup V(P_3)) \setminus \{x_1, x_2, x_3\}$, and 
\item the sets $V(P_1) \setminus \{a\}$, $V(P_2) \setminus \{a\}$ and $V(P_3) \setminus \{a\}$ are pairwise disjoint, and
\item for distinct $i,j \in \{1,2,3\}$, there are no edges between $V(P_i) \setminus \{a\}$ and $V(P_j) \setminus \{a\}$, except possibly $x_ix_j$.
\end{itemize}

\item There exists a triangle $a_1a_2a_3$ in $H$ and three paths $P_1, P_2, P_3$, where $P_i$ is from $a_i$ to $x_i$, such that
\begin{itemize}
\item $V(H) = (V(P_1) \cup V(P_2) \cup V(P_3)) \setminus \{x_1, x_2, x_3\} $, and 
\item the sets $V(P_1)$, $V(P_2)$ and $V(P_3)$ are pairwise disjoint, and
\item for distinct $i,j \in \{1,2,3\}$, there are no edges between $V(P_i)$ and $V(P_j)$, except $a_ia_j$ and possibly $x_ix_j$.
\end{itemize}
\end{enumerate}
\label{lem:three_leaves}
\end{lemma}

\begin{proof}
For some distinct $i,j,k \in  \{1,2,3\}$, let $P$ be a path from $x_i$ to $x_j$ with $V(P^*) \subseteq V(H)$ (in the graph where the edge $x_ix_j$ is deleted if it exists). Such a path exists since $x_i$ and $x_j$ have neighbors in $H$ and $H$ is connected. Assume that $x_k$ has neighbors in $P^*$. Then, by the minimality of $V(H)$, we have $V(H) = V(P^*)$. If $x_k$ has two non-adjacent neighbors in $P^*$, or $x_k$ has two neighbors in $P^*$ and its neighbors in $P^*$ are adjacent, then outcome (i) holds. If $x_k$ has a unique neighbor in $P^*$, then outcome (ii) holds. Thus, we may assume that $x_k$ is anticomplete to $P^*$.

Let $Q$ be a path with $Q \setminus \{x_k\} \subseteq V(H)$ from $x_k$ to a vertex $w \in V(H) \setminus V(P)$ (so $x_k \neq w$) with a neighbor in $P^*$. Such a path exists since $x_k$ has a neighbor in $H$, $x_k$ is anticomplete to $P^*$, and $H$ is connected. By the minimality of $V(H)$, we have $V(H) = (V(P) \cup V(Q)) \setminus \{x_1, x_2, x_3\}$ and no vertex of $Q \setminus w$ has a neighbor in $P^*$. Moreover, by the argument of the previous paragraph, we may assume  that $x_i$ and $x_j$ are anticomplete to $Q \setminus \{x_k\}$.

Now, if $w$ has a unique neighbor in $P^*$, then outcome (ii) holds. If $w$ has two neighbors in $P^*$ and its neighbors in $P^*$ are adjacent, then outcome (iii) holds. Therefore, we may assume that $w$ has two non-adjacent neighbors in $P^*$. Let $y_i$ and $y_j$ be the neighbors of $w$ in $P^*$ that are closest in $P^*$ to $x_i$ and $x_j$, respectively. Let $R$ be the subpath of $P^*$ from $y_i$ to $y_j$. Now, the graph $H'$ induced by $\left ((V(P) \cup V(Q)) \setminus V(R^*) \right) \setminus \{x_1, x_2, x_3\}$ is a connected induced subgraph of $G \setminus \{x_1, x_2, x_3\}$ and it contains at least one neighbor of $x_1$, $x_2$, and $x_3$. Moreover, $V(H')  \subset V(H)$ since $V(R^*) \neq \emptyset$. This contradicts the minimality of $V(H)$.
\end{proof}

We say that we can \emph{link a vertex $v$ onto a triangle $a_1a_2a_3$} if there exist three paths $P_1$, $P_2$, $P_3$ such that for $i=1,2,3$, $P_i$ is a path from $v$ to $a_i$, any pair of paths has only $v$ in common, for $1 \leq i < j \leq 3$, $a_ia_j$ is the unique edge of $G$ between $V(P_i)$ and $V(P_j)$, and at least two of $P_1,P_2,P_3$ are of length at least two. Observe that in a perfect graph, no vertex can be linked onto a triangle. 

Recall that for a path $P$ we denote the interior of $P$ by $P^*$. A {\em wheel} $(H, v)$ is a graph that consists of a hole $H$ and a vertex $v$ which has at least three neighbors in $H$. A {\em sector} of $(H, v)$ is a subpath $P$ of $H$ of length at least one such that $v$ is anticomplete to $P^*$ and $v$ is complete to $P \setminus P^*$. A sector is {\em short} if it is of length 1, and {\em long} otherwise. A wheel is {\em odd} if it contains an odd number of short sectors. Observe that perfect graphs do not contain odd wheels.

A wheel $(H, v)$ is a {\em universal wheel} if $v$ is complete to $H$. A {\em line wheel} is a wheel that has exactly four sectors, exactly two of which are short and the two short sectors do not have common vertices. A {\em twin wheel} is a wheel with exactly three sectors, exactly two of which are short. A {\em triangle-free wheel} is a wheel containing no triangles. A {\em proper wheel} is a wheel that is not a universal, twin, or triangle-free wheel. We need the following corollary of Lemma \ref{lem:three_leaves}.

\begin{lemma}
Let $x_1 x_2 x_3$ be a triangle in a prism-free perfect graph $G$. Assume that $H$ is a connected induced subgraph of $G \setminus \{x_1, x_2, x_3\}$ such that $H$ contains at least one neighbor of each of $x_1, x_2, x_3$, that no vertex of $H$ is complete to $\{x_1,x_2,x_3\}$, and that $V(H)$ is minimal subject to inclusion. Then, for some $\{i,j,k\}=\{1,2,3\}$, $H \cup \{x_i,x_j\}$ is a hole $H'$ and $(H',x_k)$ is a proper, or a universal, or a twin wheel. Moreover, if each of $x_1, x_2, x_3$ has a unique neighbor in $H$, then there exists $\ell \in \{i,j\}$ such that $x_{\ell}$ and $x_k$ have a common neighbor in $H$.
\label{lem:triangle}
\end{lemma}

\begin{proof}
We apply Lemma \ref{lem:three_leaves} and follow the notation in its statement. If the first outcome of Lemma~\ref{lem:three_leaves} holds, then $H \cup \{x_i,x_j\}$ is a hole, call it $H'$, and $(H',x_k)$ is a proper or a universal wheel. (Note that in this case $x_k$ has at least two neighbors in $H$, and so the second part of Lemma~\ref{lem:triangle} does not apply.) If the third outcome of Lemma~\ref{lem:three_leaves} holds, then $V(H) \cup \{x_1, x_2, x_3\}$ is a prism in $G$, a contradiction since $G$ is prism-free. Suppose the second outcome of Lemma~\ref{lem:three_leaves} holds. If at least two of the paths $P_1, P_2, P_3$ have length at least two, then we can link $a$ onto the triangle $x_1x_2x_3$ via the paths $P_1, P_2, P_3$, contradicting that $G$ is perfect. Thus, we may assume that $|V(P_1)| = |V(P_2)| = 2$. Since no vertex of $H$ is complete to $\{x_1,x_2,x_3\}$, we deduce that $a$ is not complete to $\{x_1, x_2, x_3\}$, and therefore $P_3$ has length at least two. Then, $V(H) = V(P_3) \setminus \{x_3\}$, $a \dd P_3 \dd x_3 \dd x_2 \dd a$ is a hole, call it $H'$, and $(H', x_1)$ is a twin wheel. (Note that in this case $x_1$ and $x_2$ have a common neighbor in $H$, namely $a$.)
\end{proof}

Let $(H, v)$ be a wheel that contains at least one triangle. A {\em segment} of $(H, v)$ is a maximal subpath $Q$ of $H$ such that $Q\cup \{ v\}$ is a triangle-free graph. Note that a segment may have length zero. Two segments $Q$ and $Q'$ are {\em adjacent} if $Q\cup Q' \cup \{v\}$  contains at least one triangle.

Let $(H, v)$ be a proper wheel that is not odd. A {\em bicoloring} of $(H, v)$ is a partition of vertices of $H$ into non-empty sets $R$ and $B$ (the red and blue vertices) so that the vertices in the same segment have the same color, while vertices in adjacent segments have distinct colors. The following result follows from Theorem 4.5 in \cite{CCV}.

\begin{theorem}[\cite{CCV}]
Let $G$ be a ($C_4$, prism)-free perfect graph. Let $(H, v)$ be a proper wheel in $G$ that is not a line wheel. Let $R, B$ be a bicoloring of the vertices of $H$ and assume that $B\setminus N(v)\neq \emptyset$. Then, for every $r \in R$ and $b \in B \setminus N(v)$, $v$ breaks $\{r,b\}$.
\label{thm:proper_wheels}
\end{theorem}

Next, we prove similar results for twin wheels and line wheels in ($C_4$, prism)-free perfect graphs. Let $G$ be a graph and let $A$ and $B$ be two vertex-disjoint subsets of $V(G)$ such that $A$ is anticomplete to $B$. A {\em path from $A$ to $B$} is a path $P=x_1 \dd \dots \dd x_n$ in $G \setminus (A \cup B)$ such that $x_1$ has a neighbor in $A$ and $x_n$ has a neighbor in $B$, and $P$ is inclusion-wise minimal with this property.

\begin{theorem}
Let $G$ be a ($C_4$, prism)-free perfect graph and let $(H, v)$ be a line wheel in $G$ with $H = a \dd b \dd \dots \dd c \dd d \dd \dots \dd a$ and $N(v) \cap H = \{a, b, c, d\}$. Let $P$ be the path in $H\setminus b$ from $a$ to $d$, and let $Q$ be the path in $H\setminus a$ from $b$ to $c$. Then, $v$ has a neighbor in every path from $P \setminus a$ to $Q \setminus c$ and in every path from $P \setminus d$ to $Q \setminus b$. In particular, $v$ breaks each of the sets $\{a\} \cup Q^*$, $\{d\} \cup Q^*$, $\{b\} \cup P^*$ and $\{c\} \cup P^*$.
\label{thm:line_forcers}
\end{theorem}

\begin{proof}
Let $R$ be a minimal path violating the outcome of the lemma, that is, $R$ is a path from $P \setminus a$ to $Q \setminus c$ or from $P \setminus d$ to $Q \setminus b$, $v$ is anticomplete to $R$, and $R$ is a minimal path satisfying these properties. Since $G$ is $C_4$-free perfect, the paths $P$ and $Q$ are even and of length at least four, and no vertex in $R$ is complete to $\{a, c\}$, or to $\{a, d\}$, or to $\{b, c\}$, or to $\{b, d\}$.

We recall that in a perfect graph, no vertex can be linked onto a triangle. Assume first that the path $R$ consists of a single vertex, say $R = \{r\}$. Let  $q$ be the neighbor of $r$ in $Q$ that is closest to $b$.

We claim that $r$ is anticomplete to $\{a, b, c, d\}$. Suppose $r$ is adjacent to $d$. Then, $r$ is anticomplete to $\{a, b\}$, and so $q \neq b$.  If $d$ is the unique neighbor of $r$ in $P$, then we can link $d$ onto the triangle $abv$: $d \dd v$, $d \dd P \dd a$, and $d \dd r \dd q \dd Q \dd b$. (Note that $q \neq c$ since in this case $R$ is a path from $P \setminus a$ to $Q \setminus c$.) It follows that $d$ is not the unique neighbor of $r$ in $P$. By symmetry, we deduce that $q \neq c$. Assume $r$ has two non-adjacent neighbors in $P$. Then we can link $r$ onto the triangle $abv$: $r \dd d \dd v$, $r \dd p \dd P \dd a$, and $r \dd q \dd Q \dd b$. Therefore, $r$ has two adjacent neighbors in $P$, namely $d$ and the neighbor $d'$ of $d$ in $P$. Then, we obtain a prism between the triangles $dd'r$ and $abv$, a contradiction. This proves that $r$ is not adjacent to $d$. By symmetry, it follows that $r$ is anticomplete to $\{a, b, c, d\}$.

Now, if $r$ has a unique neighbor $p$ in $P$, we can link $p$ onto the triangle $abv$: $p \dd P \dd a$, $p \dd P \dd d \dd v$, and $p \dd r \dd q \dd Q \dd b$. If $r$ has two non-adjacent neighbors $p_1, p_2$ in $P$, we can link $r$ onto $abv$: $r \dd p_1 \dd P \dd a$, $r \dd p_2 \dd P \dd d \dd v$, and $r \dd q \dd Q \dd b$. Hence, $r$ has two adjacent neighbors $p_1, p_2$ in $P$ and we obtain a prism between the triangles $rp_1p_2$ and $abv$, a contradiction. This proves that $R$ is of length at least one.

Let $R = r_1 \dd \dots \dd r_k$ with $k \geq 2$. We claim that $\{a, b, c, d\}$ is anticomplete to $R^*$. Suppose $a$ has a neighbor in $R^*$. Then, since $R$ is a minimal violating path, $b$ is the unique neighbor of $r_k$ in $Q$. By the minimality of $R$, $c$ is anticomplete to $R$ and $d$ is anticomplete to $R \setminus r_1$. Since $R$ is a violating path, $r_1$ has a neighbor in $P \setminus a$. Then, by the minimality of $R$, $b$ is anticomplete to $R \setminus r_k$. Now, we can link $b$ onto the triangle $cdv$: $b \dd v$, $b \dd Q \dd c$, and $b \dd r_k \dd R \dd r_1 \dd p \dd P \dd d$, where $p$ is the neighbor of $r_1$ in $P \setminus a$ closest to $d$ (possibly $p = d$). This proves that $\{a, b, c, d\}$ is anticomplete to $R^*$.

Next, we claim that $r_1$ is not adjacent to $b$. Suppose otherwise. Then, by the minimality of $R$, $a$ is the unique neighbor of $r_1$ in $P$, $N(r_1) \cap H=\{a,b\}$, and $R$ is a path from $P \setminus d$ to $Q \setminus b$. In particular, $r_k$ has a neighbor in $Q \setminus b$. Then, again by the minimality of $R$, $r_k$ is not adjacent to $a$. Let $q$ be the neighbor of $r_k$ in $Q$ that is closest to $c$. If $r_k$ is not adjacent to $d$, then by the minimality of $R$, $r_k$ is anticomplete to $P$ and hence we can link $a$ onto the triangle $cdv$: $a \dd v$, $a \dd P \dd d$, and $a \dd r_1 \dd R \dd r_k \dd q \dd c$. So, $r_k$ is adjacent to $d$. This restores the symmetry between $r_1$ and $r_k$, and so $N(r_k) \cap H=\{c,d\}$ and $P \cup Q \cup R$ is a prism in $G$, a contradiction. This proves that $r_1$ is not adjacent to $b$. By symmetry, $r_1$ is not adjacent to $c$, and $r_k$ is not adjacent to $a, d$. So by the minimality of $R$, $R \setminus r_k$ is anticomplete to $Q$ and $R \setminus r_1$ is anticomplete to $P$.

Now, assume that $r_1$ has a unique neighbor $p$ in $P$. We may assume that $p \in P \setminus a$, and so we can link $p$ onto the triangle $abv$: $p \dd P \dd a$, $p \dd P \dd d \dd v$, and $p \dd r_1 \dd R \dd r_k \dd q \dd Q \dd b$, where $q$ is the neighbor of $r_k$ in $Q$ closest to $b$, unless $q=c$. If $q=c$, then $p \neq d$, and hence we can link $c$ onto the triangle $abv$. This proves that $r_1$ has at least two neighbors in $P$, and by symmetry, $r_k$ has at least two neighbors in $Q$. Suppose $r_1$ has two non-adjacent neighbors $p_1, p_2$ in $P$. Then, we can link $r_1$ onto the triangle $abv$: $r_1 \dd p_1 \dd P \dd a$, $r_1 \dd p_2 \dd P \dd d \dd v$, and $r_1 \dd R \dd r_k \dd q \dd Q \dd b$, where $q$ is the neighbor of $r_k$ in $Q$ closest to $b$ (possibly $q=b$ but certainly $q \neq c$ since $r_k$ has at least two neighbors in $Q$). This proves that $r_1$ has two adjacent neighbors $p_1, p_2$ in $P$, and by symmetry, $r_k$ has two adjacent neighbors $q_1, q_2$ in $Q$. Then, $H \cup R$ is a prism between the triangles $r_1p_1p_2$ and $r_kq_1q_2$, a contradiction.
\end{proof}

\leqnomode
\begin{theorem}
Let $G$ a ($C_4$, prism)-free perfect graph and let $(H,v)$ be a twin wheel in $G$. Let $a,u,b$ be the neighbors of $v$ in $H$ appearing in this order when traversing $H$. Let $S_u=N[u]\setminus \{v\}$ and $S_v=N[v]\setminus \{u\}$. Then, one of the following holds:
\begin{enumerate}[(i)]
\itemsep0em
\item $S_u$ is a cutset of $G$ that separates $v$ from a vertex of $H\setminus \{a, u, b\}$,
\item $S_v$ is a cutset of $G$ that separates $u$ from a vertex of $H\setminus \{a, u, b\}$.
\end{enumerate}
In particular, either $u$ breaks $H\setminus \{a,b\}$ or $v$ breaks $H\setminus \{a,b\}$.
\label{thm:twin_forcers}
\end{theorem}

\begin{proof}
Suppose not. Let $P=p_1 \dd \ldots \dd p_k$ be a path from $v$ to $H\setminus \{ a,u,b\}$ in $G\setminus S_u$, and $Q=q_1 \dd \ldots \dd q_l$ be a path from  $u$ to $H\setminus \{ a,u,b\}$ in $G\setminus S_v$. In particular, $p_1$ is the unique neighbor of $v$ in $P\cup Q\cup (H\setminus \{ a,u,b\})$, and $q_1$ is the unique neighbor of $u$ in $P\cup Q\cup (H\setminus \{ a,u,b\})$, and $p_1\neq q_1$. Since $G$ is $C_4$-free, $p_1q_1$ is not an edge, and $a$ and $b$ do not have a common neighbor in $P\cup Q\cup (H\setminus \{ a,u,b\})$. Note that $P\cup Q\cup (H\setminus \{ a,u,b\})$ is connected. Let $T$ be a path in $P\cup Q\cup (H\setminus \{ a,u,b\})$ from $p_1$ to $q_1$, and let $H'$ be the hole induced by $T\cup \{u, v\}$.

\vspace{-0.3cm}

\begin{equation}
\longbox{{\it Both $a$ and $b$ have a neighbor in $T$. In particular, both $(H',a)$ and $(H',b)$ are proper or twin wheels.}}
\label{eq:a_b_nhbr_in_T}\tag{$5.5.1$}
\end{equation}
{\em Proof of \eqref{eq:a_b_nhbr_in_T}}:
Assume that $a$ has no neighbor in $T$. Let $S$ be a path from $a$ to $T$ with $S \subseteq (H \cup P \cup Q) \setminus \{u,v\}$. Then, $(T \cup S) \setminus \{u,v,a\}$ is a connected graph, containing exactly one neighbor of  each of $u,v,a$, and no vertex of $(T \cup S) \setminus \{u,v,a\}$ is adjacent to two of $u,v,a$, contrary to Lemma \ref{lem:triangle}. This proves \eqref{eq:a_b_nhbr_in_T}.

\vspace{-0.3cm}

\begin{equation}
\mbox{{\it Both $(H',a)$ and  $(H',b)$ are twin  wheels.}}
\label{eq:they_are_twin_w}\tag{$5.5.2$}
\end{equation}
{\em Proof of \eqref{eq:they_are_twin_w}}:
Suppose $(H',a)$ is a proper wheel. Consider a bicoloring $R, B$ of $(H',a)$. We may assume that $u\in R$, $v\in B$ and $B\setminus N(a)\neq \emptyset$. By Theorem \ref{thm:proper_wheels} and Theorem \ref{thm:line_forcers}, for every $x \in B \setminus N(a)$, $a$ breaks $\{u,x\}$. Since $b$ is adjacent to $u$, it follows that all the neighbors of $b$ in $H'\setminus \{v\}$ belong to $R$. In particular, by \eqref{eq:a_b_nhbr_in_T}, $R\setminus N(a)\neq \emptyset$. So, by Theorem \ref{thm:proper_wheels} and Theorem \ref{thm:line_forcers}, for every $y \in R \setminus N(a)$, $a$ breaks $\{v, y\}$. This contradicts the fact that $b$ is adjacent to $v\in B$ and has a neighbor in $R\setminus N(a)$. This proves \eqref{eq:they_are_twin_w}.

\medskip

By \eqref{eq:they_are_twin_w}, and since $a$ and $b$ do not have a common neighbor in $H'\setminus \{ u,v\}$, we may assume that $N(a)\cap H'=\{ p_1,u,v\}$ and $N(b)\cap H'=\{ u,v,q_1\}$. If no vertex of $H \setminus \{ a,u,b\}$ has a neighbor in $T$ (and in particular $T\cap H=\emptyset$), then we can link $b$ onto the triangle $p_1ua$ in $(H\setminus \{u\}) \cup T$. Thus, some vertex of $H\setminus \{ a,u,b\}$ has a neighbor in $T$, and so we may assume that $p_k\in T$, i.e. $P\subseteq T$. It follows that $N(a)\cap P=\{ p_1\}$. Let $a'$ be the neighbor of $a$ in $H \setminus u$. Suppose first that $p_k$ has a neighbor in $H \setminus a'$. Then, there is a path $S$ from $p_1$ to $b$ with $S^* \subseteq (H \cup P) \setminus \{a, a', b, u, v\}$, But $a \dd p_1 \dd S \dd b \dd v \dd a$ is a hole and $a$ is an odd wheel center for this hole, a contradiction. This proves that $a'$ is the unique neighbor of $p_k$ in $H \setminus \{a\}$. If $k>1$, then we can link $a'$ onto the triangle $avp_1$ in the graph $(H \setminus \{u\})\cup P \cup Q$, so $k=1$.
Since $(H',a)$ is a twin wheel, it follows that $a' \not \in H'$, and so $p_1$ has a neighbor in $Q$. Let $i$ be minimum such that $p_1$ is adjacent to $q_i$. If $i<l$, then $H \cup \{p_1, q_1, \dots, q_i\}$ induces a prism with triangles $aa'p_1$ and $ubq_1$, a contradiction.
This proves that $i=l$, and consequently $q_l \in T$ and $Q \subseteq T$. This restores the symmetry between $P$ and $Q$, and therefore $l=1$, and $p_1q_1$ is an edge, a contradiction.
\end{proof}

We now prove that paws are forcers for ($C_4$, prism)-free perfect graphs. (It is easy to verify that paws may not be forcers for graphs that are not $C_4$-free, prism-free, or perfect.)

\begin{theorem}
Let $G$ be a ($C_4$, prism)-free perfect graph. Let $P$ be a paw in $G$ with vertex set $\{a, c, b_1, b_2\}$ and edge set $\{ab_1, ab_2, b_1b_2, ac\}$. Then, either $a$ breaks $\{c, b_1,b_2\}$, or $b_1$ breaks $\{c, b_2\}$, or $b_2$ breaks $\{c, b_1\}$.
\label{thm:paws_are_forcers}
\end{theorem}

\begin{proof}
Suppose $a$ does not break $\{c, b_1, b_2\}$. Then, there exists a connected component $D$ of $G \setminus N[a]$ such that $\{c,b_1,b_2\} \subseteq N[D]$. Since $a$ is adjacent to $c, b_1, b_2$, it follows that $c, b_1, b_2 \not \in D$. Then, $D' = D \cup \{c\}$ is a connected graph containing neighbors of each of $a, b_1, b_2$, and $c$ is the unique neighbor of $a$ in $D'$. Let $H$ be a connected induced subgraph of $D'$ such that $H$ contains at least one neighbor of each of $a, b_1, b_2$, and that $V(H)$ is minimal subject to inclusion. Note that $c \in H$ and that no vertex of $H$ is complete to $\{a, b_1, b_2\}$. Then, by Lemma \ref{lem:triangle}, either
\begin{itemize}
\itemsep0em
\item $H \cup \{b_1, b_2\}$ is a hole $H'$ and $(H', a)$ is a proper, or a universal, or a twin wheel; or
\item (possibly switching the roles of $b_1$ and $b_2$), $H \cup \{a, b_1\}$ is a hole $H'$ and $(H', b_2)$ is a proper, or a universal, or a twin wheel.
\end{itemize}
The first outcome does not hold since in that case $N(a) \cap H' = \{c,b_1,b_2\}$ leading to an odd wheel, which contradicts $G$ being perfect. Therefore, the second outcome holds. Now, since $b_2$ is not adjacent to $c$, $(H', b_2)$ is a proper or a twin wheel. If $(H', b_2)$ is a proper wheel that is not a line wheel or if $(H', b_2)$ is a line wheel, then $b_2$ breaks $\{c, b_1\}$ by Theorem \ref{thm:proper_wheels} and Theorem \ref{thm:line_forcers}, respectively. If $(H', b_2)$ is a twin wheel, then by Theorem \ref{thm:twin_forcers}, either $b_2$ breaks $\{c, b_1\}$ or $b_1$ breaks $\{c, b_2\}$.
\end{proof}

A graph $G$ is {\em paw-friendly} if $G$ is perfect and for every induced subgraph $H$ of $G$,
if $P$ is a paw in $H$ with vertex set $\{c, a, b_1, b_2\}$ and edge set
$\{ab_1, ab_2, b_1b_2, ac\}$, then with respect to $H$ either $a$ breaks $\{c, b_1,b_2\}$, or $b_1$ breaks $\{c, b_2\}$, or $b_2$ breaks $\{c, b_1\}$. It follows from this definition that paws are forcers for paw-friendly graphs. The following is immediate from Theorem \ref{thm:paws_are_forcers}.

\begin{theorem}
  \label{thm:C4prismpawfriendly}
  Every ($C_4$, prism)-free perfect graph is paw-friendly.
\end{theorem}

\section{The star-free bag of paw-friendly graphs}
\label{sec:star_free_bag} 

  In this section, we prove several properties of the star-free bag of paw-friendly graphs. All of the definitions and lemmas in this section share the following assumptions: \\
  
  \noindent {\bf Common assumptions for Section 6:} {\em Let $c \in [\frac{1}{2}, 1)$ and let $d, \delta$ be positive integers with $d \geq \delta + 3$. Let G be a connected paw-friendly graph with maximum degree $\delta$, and $w:V(G) \to [0, 1]$ a
weight function on $V(G)$ with $w(G) = 1$. Assume that $G$ has no $d$-bounded $(w, c)$-balanced separator. Let $X$ be an $\smallO$-star cover
of $G$, $\S$ an $\smallO$-star covering sequence of $G$, and $\beta$ the star-free bag of $G$.} 

\begin{lemma}
\label{lemma:beta_bipartite}
\label{lemma:dim_is_2}
 The following hold:
 \begin{enumerate}[(i)]
\item $\beta = X$ and $\beta$ is bipartite.
\item Let $(X_1, X_2)$ be a bipartition of $\beta$. Then for $i = 1, 2$, $\S_i = \{S_x \mid x \in X_i\}$ is loosely laminar.
\end{enumerate}
\end{lemma}
\begin{proof}
By Lemma \ref{lemma:dimS}, $k = \dim(\S) \leq \delta + 1$, and so $d \geq k + 2$.  By Lemma \ref{lemma:star_free_bag}, $\beta$ is connected, and by Lemma \ref{lemma:bag_no_forcers}, $\beta$ is paw-free since paws are forcers for $G$ by definition. Every connected paw-free graph is either triangle-free or complete multipartite \cite{O-paw-free}. Suppose $\beta$ is complete multipartite. Then, there exists $v \in \beta$ such that $\beta \subseteq N^2[v]$, so $\beta$ is a 2-bounded $(w_{\beta}, c)$-balanced separator for $\beta$, where $w_\beta$ is defined in Lemma \ref{lemma:star_free_bag}. However, by Lemma \ref{lemma:star_free_bag}, $\beta$ does not have a $(d-k)$-bounded $(w_\beta, c)$-balanced separator, a contradiction since $d \geq k + 2$. Therefore, $\beta$ is triangle-free. Since $\beta$ is perfect, it follows that $\beta$ is bipartite. By Lemma \ref{lemma:beta_is_X}, $\beta = X$. This proves (i). 

Let $\S_1, \S_2$ be as in (ii). By Lemma \ref{lemma:loosely_noncrossing}, $\S_1$ and $\S_2$ are loosely laminar. This proves (ii). 
\end{proof}

For the remainder of the paper, if $G$ is a connected paw-friendly graph with no $d$-bounded $(w, c)$-balanced separator and $\beta$ is the star-free bag of $G$, we let $(X_1, X_2)$ be a bipartition of $\beta$, and we let $(\S_1, \S_2)$ be the corresponding partition of $\S$ into loosely laminar sequences; i.e. $\S_i = \{S_x : x \in X_i\}$. (Note that this implies that $\dim(\S) \leq 2$.) Let $\gamma$ denote the central bag for $\S_2$, i.e.,
\begin{equation*}
\hspace{5cm}
\gamma = \bigcap_{S \in \S_2} (B(S) \cup C(S)).
\end{equation*}

Since $\beta = \bigcap_{S \in \S_1 \cup \S_2} (B(S) \cup C(S))$, we have $\beta \subseteq \gamma \subseteq G$. Therefore, $\gamma$ is called the {\em intermediate bag} of $G$. 
In the following lemma, we summarize several important properties of $\gamma$ and $\beta$.

\begin{lemma}
\label{lemma:properties}
Let $\gamma$ be the intermediate bag of $G$. Then,
\begin{enumerate}[(i)]
\itemsep -0.2em
    \item $\gamma$ and $\beta$ are connected,
    \item $\gamma$ has no $(d-1)$-bounded $(w_{\S_2}, c)$-balanced separator and $\beta$ has no $(d-2)$-bounded $(w_{\beta}, c)$-balanced separator,
    \item $G \setminus \gamma \subseteq \bigcup_{x_2 \in X_2} A_{x_2}$ and $ \gamma \setminus \beta \subseteq \bigcup_{x_1 \in X_1} A_{x_1}$. 
\end{enumerate}
\end{lemma}
\begin{proof}
Since $\gamma$ is the central bag for $\S_2$, it follows from Lemma \ref{lemma:central_bag_1} that $\gamma$ is connected. By Lemma~\ref{lemma:star_free_bag}, $\beta$ is connected. Moreover, by Lemma \ref{lemma:central_bag_2}, $\gamma$ has no $(d-1)$-bounded $(w_{\S_2}, c)$-balanced separator, and by Lemma \ref{lemma:star_free_bag}, $\beta$ has no $(d-2)$-bounded $(w_{\beta}, c)$-balanced separator. Finally, since $\gamma$ is the central bag for $\S_2$ and $\beta$ is the central bag in $\gamma$ for $\S_1\mid \gamma = \{S \cap \gamma: S \in \S_1\}$, property (iii) holds. 
\end{proof}

In the next section, $\gamma$ and $\beta$ are integral to proving structural results about $G$. We also define the {\em core bag} of $G$, denoted $\R$, as follows:
\begin{equation*}
\hspace{5cm}
\R = \beta \cup \left(\bigcup_{x_2 \in X_2} C_{x_2}\right). 
\end{equation*}

\begin{figure}[h]
\centering
    \begin{tikzpicture}
        \draw[rounded corners, pattern=crosshatch dots, pattern color = gray!50] (0.75, 2) rectangle (-0.75,-2);

        \draw[rounded corners, pattern=crosshatch dots, pattern color = gray!50] (2, 1.9) rectangle (1.2, 1.1);

        \draw[rounded corners, pattern=crosshatch dots, pattern color = gray!50] (2, 0.9) rectangle (1.2, 0.1);

        \draw[rounded corners, pattern=crosshatch dots, pattern color = gray!50] (2, -0.1) rectangle (1.2, -0.9);

        \draw[rounded corners, pattern=crosshatch dots, pattern color = gray!50] (2, -1.1) rectangle (1.2, -1.9);

        \draw[rounded corners, pattern=crosshatch dots, pattern color = gray!50] (3.25, 1.9) rectangle (2.45, 1.1);

        \draw[rounded corners, pattern=crosshatch dots, pattern color = gray!50] (3.25, 0.9) rectangle (2.45, 0.1);

        \draw[rounded corners, pattern=crosshatch dots, pattern color = gray!50] (3.25, -0.1) rectangle (2.45, -0.9);

        \draw[rounded corners, pattern=crosshatch dots, pattern color = gray!50] (3.25, -1.1) rectangle (2.45, -1.9);

        \draw (0, 2)--(0, -2); 

        \coordinate[label={$X_1$}] (E) at (-.375,-2.6);

        \coordinate[label={$X_2$}] (E) at (.375,-2.6);

        \coordinate[label={$\gamma \setminus \beta$}] (E) at (1.6,-2.6);

        \coordinate[label={$G \setminus \gamma$}] (E) at (2.85,-2.6);

        \node[draw, circle, inner sep=2pt, fill=black] at (.375, 1) (1) {};   

        \node[draw, circle, inner sep=2pt, fill=black] at (1.4, 1.5) (2) {}; 

        \node[draw, circle, inner sep=2pt, fill=black] at (1.4, 0.4) (3) {};   
        \node[draw, circle, inner sep=2pt, fill=black] at (.375, -1) (4) {};   

        \node[draw, circle, inner sep=2pt, fill=black] at (1.4, -1.5) (5) {}; 

        \node[draw, circle, inner sep=2pt, fill=black] at (1.4,- 0.4) (6) {};   

        \draw (1)--(2); 
        \draw (1)--(3);
        \draw (4)--(5); 
        \draw (4)--(6);
    
    \end{tikzpicture}
    \caption{A drawing showing the central, core, and intermediate bags. Recall that $G \setminus \gamma = \bigcup_{x \in X_2} A_{x}$, $\gamma \setminus \beta \subseteq \bigcup_{x \in X_1} A_x$, $\mathcal{R} = \beta \cup \bigcup_{x \in X_2} C_{x}$, and $\beta = X_1 \cup X_2$.}
\end{figure}

The following lemma contains two useful facts about the composition of $\R$.

\begin{lemma}
Let $\gamma$ be the intermediate bag of $G$ and let $\R$ be the core bag of $G$. Then, 
\begin{enumerate}[(i)]
\itemsep -0.2em
    \item  $\R \subseteq \gamma$, and

    \item every connected component of $G \setminus \R$ is a subset of $A_x$ for some $x \in \beta$.
\end{enumerate}
\label{lemma:structure_of_R}
\end{lemma}
\begin{proof}
Recall that $(X_1, X_2)$ is the partition of $\beta$ into independent sets. By definition, $\beta \subseteq \gamma$. Further, by Lemma \ref{lemma:central_bag_1}, $C_{x_2} \subseteq \gamma$ for all $x_2 \in X_2$. Therefore, $\R \subseteq \gamma$. This proves (i).

Now, we prove (ii). Let $D$ be a connected component of $G \setminus \R$. Suppose that for some $x_2 \in X_2$, $D \cap A_{x_2} \neq \emptyset$, and let $D'$ be a connected component of $A_{x_2}$ such that $D' \cap D \neq \emptyset$. If $D \not \subseteq D'$, then there exists $v \in D \setminus D'$ that has a neighbor in $D'$. Then, $v \in C_{x_2}$, but $C_{x_2} \subseteq \R$, a contradiction. So $D \subseteq D'$, and (ii) holds.

Thus, we may assume that $D \cap A_{x_2} = \emptyset$ for all $x_2 \in X_2$. Since by Lemma \ref{lemma:properties} $G \setminus \gamma \subseteq \bigcup_{x_2 \in X_2} A_{x_2}$, it follows that $D \subseteq \gamma \setminus \beta$. Since $\gamma \setminus \beta \subseteq \bigcup_{x \in X_1}A_x$ and since by Lemma \ref{lemma:dim_is_2} $\mathcal{S}_1$ is loosely laminar, it follows that $D \subseteq A_{x_1}$ for some $x_1 \in X_1$.
\end{proof}

Note that (ii) of Lemma \ref{lemma:structure_of_R} shows that every connected component $D$ of $G \setminus \R$ satisfies $|N(D)| \leq d$ and $|D| \leq cn$. Therefore, if $\R$ is the union of iterated even sets, $\R$ is an even set separator of $G$. In the next section, we prove that $\R$ is the union of iterated even sets.    

\section{The core bag $\R$ is a $(\delta^2+2)$-iterated even set}
\label{sec:R_evensetseparator}

In this section, we prove that if $G$ is a paw-friendly graph with bounded degree, then the core bag $\R$ of $G$ is a $(\delta^2 + 2)$-iterated even set in $G$. Let $(X_1, X_2)$ be the bipartition of $\beta$. Recall that $\R = \beta \cup \left(\bigcup_{x_2 \in X_2} C_{x_2}\right)$. By Lemma \ref{lemma:beta_bipartite}, $\beta$ is bipartite. In Section \ref{sec:even_sets_in_beta}, we show that $X_1$ and $X_2$ are even sets in $G$. In Section \ref{sec:even_sets_Rminusbeta}, we show that $\R \setminus \beta = \bigcup_{x_2 \in X_2} C_{x_2}$ is a $\delta^2$-iterated even set in $G \setminus \beta$.

We begin with three important lemmas.

\begin{lemma}
\label{lemma:even_paths}
Let $c \in [\frac{1}{2}, 1)$ and let $d$ be a positive integer. Let $G$ be an odd-hole-free graph with no $d$-bounded $(w, c)$-balanced separator. Let $v \in V(G)$ and let $S_v = (A_v, C_v, B_v)$ be the canonical star separation for $v$. Let $c_1, c_2 \in C_v$ be such that $c_1$ is not adjacent to $c_2$, and let $P$ be a path from $c_1$ to $c_2$ with $P^* \subseteq B_v$ or $P^* \subseteq A_v$. Then, $P$ is even. 
\end{lemma}
\begin{proof}
Since $c_1$ is not adjacent to $c_2$, $\{c_1, c_2\} \subseteq C_v \setminus \{v\}$. Suppose $P^* \subseteq B_v$. Since $v$ is anticomplete to $B_v$, it follows that $v \dd c_1 \dd P \dd c_2 \dd v$ is a hole of $G$. Since $G$ is perfect, every hole in $G$ is even. Therefore, $P$ is even. 

Now, suppose $P^* \subseteq A_v$. Let $Q$ be a path from $c_1$ to $c_2$ through $B_v$. The path $Q$ exists since $c_1$ and $c_2$ have neighbors in $B_v$ and $B_v$ is connected. Note that $Q$ is even. Then, since $c_1 \dd P \dd c_2 \dd Q \dd c_1$ is an even hole of $G$, it follows that $P$ is even. 
\end{proof}

Let $v \in V(G)$ and let $P = p_1 \dd \hdots \dd p_k$ be a path in $G \setminus \{v\}$ such that $p_1, p_k \in B_v \cup C_v$. Let $i$ be minimum and $j$ be maximum such that $v$ is adjacent to $p_i$ and $p_j$. Note that since $p_1, p_k \in B_v \cup C_v$, we have $p_i, p_j \in C_v$. The {\em span of $v$ in $P$} is the subpath $p_i \dd P \dd p_j$. If the span of $v$ in $P$ has odd length greater than one, we say that $v$ has {\em wide odd span in $P$}; otherwise, if the span is greater than one, we say that $v$ has {\em even span in $P$}.

\begin{lemma}
\label{lemma:exists_a_paw}
Let $c \in [\frac{1}{2}, 1)$ and let $\delta, d$ be positive integers with $d > 2$. Let $G$ be a connected paw-friendly graph with maximum degree $\delta$ and no $d$-bounded $(w, c)$-balanced separator. Let $\S$ be a loosely laminar sequence of canonical star separations, let $\beta_\S$ be the central bag for $\S$, and let $H$ be an induced subgraph of $G$. Let $v_1, v_2 \in \beta_\S \cap V(H)$ be an even pair in $\beta_\S \cap H$, and let $P = p_1 \dd \hdots \dd p_k$ be an odd path  from $v_1$ to $v_2$ in $H$, with $p_1 = v_1$ and $p_k = v_2$. Suppose $u \in V(G)$ is such that $S_u \in \S$ and $u$ has a wide odd span in $P$. Then, there exist $p_q, p_r, p_s \in P \cap C_u$ with $q < s - 2$ and either $r =  q + 1$ or $r = s-1$, such that $\{u, p_q, p_r, p_s\}$ is a paw of $G$ with edge set $\{up_q, up_r, up_s, p_qp_r\}$ or $\{up_q, up_r, up_s, p_rp_s\}$.
\end{lemma}

\begin{proof}
Let $p_i \dd P \dd p_j$ be the span of $u$ in $P$. Since $u$ has odd span in $P$, it follows that $j-i$ is odd. By Lemma \ref{lemma:even_paths}, every path between two vertices of $C_{u}$ through $A_{u}$ or through $B_{u}$ is even. Therefore, either there exist $q \in \{i,\ldots,j-3\}$ and $r \in \{q+3, \hdots, j\}$ such that $p_q,p_{q+1},p_r \in  C_{u} \cap H$, and $p_{q+2}, \ldots,
p_{r-1} \not \in C_{u} \cap H$, or there exist $q \in \{i+2, \ldots, j-1\}$ and $r \in \{i, \ldots, q-2\}$ such that $p_q, p_{q+1}, p_r \in C_{u} \cap H$ and $p_{r+1}, \hdots, p_{q-1} \not \in C_u \cap H$. 
\end{proof}

\begin{lemma}
Let $c \in [\frac{1}{2}, 1)$ and let $d$ be a positive integer. Let $G$ be a paw-friendly graph with no $d$-bounded $(w, c)$-balanced separator. Let $x, y \in V(G)$ be such that $x \in B_y \cup C_y$ and $y \in B_x \cup C_x$. Let $P = p_1 \dd \hdots \dd p_k$ be a path from $x$ to $y$ (so $p_1 = x$ and $p_k = y$), and suppose $p_q, p_r, p_s \in P$ and $u \in V(G) \setminus P$ are such that $p_q, p_r, p_s \in C_{u}$, $q < s - 2$, and either $r = q + 1$ or $r = s - 1$. Then, either $p_r \leq_A x$ or $p_r \leq_A y$.
\label{lemma:key_lemma_paw_shield}
\end{lemma}
\begin{proof}
Note that $\{u, p_q, p_r, p_s\}$ is a paw with edge set $\{up_q, up_r, up_s, p_qp_r\}$ or $\{up_q, up_r, up_s, p_rp_s\}$. Since $p_q, p_r, p_s \in C_u$, $u$ does not break $\{p_q, p_r, p_s\}$. Suppose $r = q+1$. Then, $p_r \dd P \dd p_s$ is a path from $p_r$ to $p_s$ such that $p_{r+1} \dd P \dd p_s$ is anticomplete to $p_q$, so $p_q$ does not break $\{p_r, p_s\}$. Therefore, $p_r$ is a center of the paw and $p_r$ breaks $\{p_q, p_s\}$. Now, suppose $r = s-1$. Then, $p_q \dd P \dd p_r$ is a path from $p_q$ to $p_r$ anticomplete to $p_s$, so $p_s$ does not break $\{p_q, p_r\}$. Therefore, $p_r$ is a center of the paw and $p_r$ breaks $\{p_q, p_s\}$. We may assume by symmetry that $p_q \in A_{p_r}$. Then, $x \dd P \dd p_q$ is a path from $x$ to $p_q$ not going through neighbors of $p_r$, so $x \in A_{p_r}$. By Lemma \ref{lemma:star_shields}, either $p_r$ is a shield for $x$ or $x$ and $p_r$ are star twins, so we may assume $x$ and $p_r$ are star twins, for otherwise Lemma \ref{lemma:key_lemma_paw_shield} holds. Consequently, $p_r \in A_x$. Now, $p_r \dd P \dd y$ is a path from $p_r$ to $y$ not going through neighbors of $x$, so $y \in A_x$, a contradiction. 
\end{proof}

The lemmas in the remainder of Section 7 share the same common assumptions as the lemmas in Section 6. \\

\noindent \textbf{Common assumptions for Sections 7.1 and 7.2:} {\em Let $c \in [\frac{1}{2}, 1)$ and let $d, \delta$ be positive integers with $d \geq \delta + 3$. Let $G$ be a connected paw-friendly graph with maximum degree $\delta$, and $w: V(G) \to [0, 1]$ a weight function on $V(G)$ with $w(G) = 1$. Assume that $G$ has no $d$-bounded $(w, c)$-balanced separator. Let $X$ be an $\smallO{}$-star cover of $G$, $\S$ an $\smallO$-star covering sequence of $G$, and $\beta$ the star-free bag of $G$.}
\subsection{Even sets in $\beta$}
\label{sec:even_sets_in_beta}

Let $(X_1, X_2)$ be the bipartition of $\beta$. In this section, we prove that $X_1$ and $X_2$ are even sets in $G$. First, we show that $X_1$ and $X_2$ are even sets in $\gamma$. 

\begin{lemma}
Let $(X_1, X_2)$ be the bipartition of $\beta$. Then, $X_1$ and $X_2$ are even sets in $\gamma$.
\label{lemma:T_even_beta}
\end{lemma}
\begin{proof}
First, let $t_1, t_1' \in X_1$, and suppose that $P = p_1 \dd \hdots \dd p_k$ is an odd path from $t_1$ to $t_1'$ in $\gamma$ with $p_1 = t_1$ and $p_k = t_1'$. We assume $P$ is chosen with $|V(P) \cap (\gamma \setminus \beta)|$ minimum. Note that $\beta = \beta_{\S_1} \cap \gamma$, where $\beta_{\S_1}$ is the central bag for $\S_1$.

We claim that $C_{x_1} \cap \gamma \subseteq \beta$ for every $x_1 \in X_1$. By (iii) of Lemma \ref{lemma:properties}, we have $\gamma \setminus \beta \subseteq \bigcup_{S \in \S_1}A(S)$. Since $\S_1$ is loosely laminar, it follows that $C_{x_1} \cap A(S) = \emptyset$ for all $x_1 \in X_1$ and $S \in \S_1$. Therefore, $C_{x_1} \cap (\gamma \setminus \beta) = \emptyset$, and $C_{x_1} \cap \gamma \subseteq \beta$ for all $x_1 \in X_1$. This proves the claim.

Next, we show that the span of $x_1$ in $P$ is wide odd for some $x_1 \in X_1$. Since $\beta$ is bipartite, $P \cap (\gamma \setminus \beta) \neq \emptyset$, and therefore $P \cap A_{x_1} \neq \emptyset$ for some $x_1 \in X_1$. Note that it follows that the span of $x_1$ in $P$ is of length greater than one. Since $t_1, t_1' \in \beta$, it follows that $t_1, t_1' \in C_{x_1} \cup B_{x_1}$. Suppose $x_1$ has an even span $p_i \dd P \dd p_j$ in $P$. Then, $P' = p_1 \dd P \dd p_i \dd x_1 \dd p_j \dd P \dd p_k$ is an odd path from $t_1$ to $t_1'$, and $|V(P') \cap (\gamma \setminus \beta)| < |V(P) \cap (\gamma \setminus \beta)|$, a contradiction. It follows that the span of $x_1$ in $P$ is wide odd for some $x_1 \in X_1$.

By Lemma \ref{lemma:exists_a_paw}, $P$ contains an edge $p_qp_{q+1}$ such that $p_q, p_{q+1} \in C_{x_1} \cap \gamma$. But since $C_{x_1} \cap \gamma \subseteq \beta$ and $\beta$ is bipartite, it follows that $(C_{x_1} \setminus \{x_1\}) \cap \gamma$ is independent, a contradiction. Therefore, $X_1$ is an even set in $\gamma$, and by an analogous argument, $X_2$ is an even set in $\gamma$. 
\end{proof}

Now, we prove that $X_1$ and $X_2$ are even sets in $G$. 

\begin{lemma}\label{lemma:T1T2_even}
Let $(X_1, X_2)$ be the bipartition of $\beta$. Then, $X_1$ and $X_2$ are even sets in $G$. 
\end{lemma}

\begin{proof}
Let $t_1, t_1' \in X_1$, and let $P = p_1 \dd \hdots \dd p_k$ be an odd path in $G$ with $p_1 = t_1$ and $p_k = t_1'$. Assume $P$ is chosen with $|V(P) \cap (V(G) \setminus \gamma)|$ minimum.

We claim that the span of $x_2$ is wide odd for some $x_2 \in X_2$. By Lemma \ref{lemma:T_even_beta}, $t_1$ and $t_1'$ are an even pair in $\gamma$, so $P \cap (V(G) \setminus \gamma) \neq \emptyset$. Since $V(G) \setminus \gamma \subseteq \bigcup_{x_2 \in X_2} A_{x_2}$, it follows that $P \cap A_{x_2} \neq \emptyset$ for some $x_2 \in X_2$. Note that it follows that the span of $x_2$ in $P$ is of length greater than one. Since $t_1, t_1' \in \beta$ (as $\beta = X_1 \cup X_2$), it follows that $t_1, t_1' \in B_{x_2} \cup C_{x_2}$. Suppose $x_2$ has even span $p_i \dd P \dd p_j$ in $P$. Then, $P' = p_1 \dd P \dd p_i \dd x_2 \dd p_j \dd P \dd p_k$ is an odd path from $t_1$ to $t_1'$ and $|V(P') \cap (V(G) \setminus \gamma)| < |V(P) \cap (V(G) \setminus \gamma)|$, a contradiction. This proves the claim.

By Lemma \ref{lemma:exists_a_paw}, there exists a paw $\{u, p_q, p_r, p_s\}$, such that $u \in X_2$, $p_q, p_r, p_s \in P$, $q < s-2$, either $r = q+1$ or $r = s-1$. Since $t_1, t_1' \in \beta$, it follows that $t_1 \in B_{t_1'} \cup C_{t_1'}$ and $t_1' \in B_{t_1} \cup C_{t_1}$. Therefore, by Lemma \ref{lemma:key_lemma_paw_shield}, either $p_r \leq_A t_1$ or $p_r \leq_A t_1'$. But by Lemma \ref{lemma:beta_is_X}, $\beta$ is exactly the set of minimal vertices under the $\leq_A$ order, a contradiction. 
\end{proof}

\subsection{Even sets in $\R \setminus \beta$} 
\label{sec:even_sets_Rminusbeta}

In this section, we construct a $\delta^2$-iterated even set $(L_1, \hdots, L_{\delta^2})$ of $G \setminus \beta$ such that $L_1 \cup \hdots \cup L_{\delta^2} = \R \setminus \beta$. Similarly to the approach in Section \ref{sec:even_sets_in_beta}, we first show that $(L_1, \hdots, L_{\delta^2})$ is a $\delta^2$-iterated even set in $\gamma \setminus \beta$. We start with a lemma. 

\begin{lemma}
Let $(X_1, X_2)$ be the bipartition of $\beta$ and let $\R$ be the core bag of $G$. Let $x_1 \in X_1$ and let $D_{x_1} = A_{x_1} \cap \R$. Then, $|D_{x_1}| \leq \delta^2$. 
\end{lemma}

\begin{proof}
Let $v \in D_{x_1}$.  Since $v \in \R \setminus \beta$, it follows that $v \in C_{x_2}$ for some $x_2 \in X_2$. Since $x_2$ has a neighbor in $A_{x_1}$ and $x_2 \in C_{x_1} \cup B_{x_1}$, we have $x_2 \in C_{x_1}$, and so $x_1$ is adjacent to $x_2$. Therefore, $D_{x_1} \subseteq N^2[x_1]$. Since $G$ has maximum degree $\delta$, it follows that $|D_{x_1}| \leq \delta^2$.
\end{proof}

In Lemma \ref{lemma:T1T2_even}, the key contradiction comes from the fact that when $P$ is an odd path from $x$ to $y$, there exists a vertex $p_i \in P$ such that $p_i \in A_x$. To reach that contradiction in this case, we make use of the relation $\leq_A$.  Recall the relation $\leq_A$ defined on $V(G)$,
\begin{equation*}
\hspace{2.5cm}
x \leq_A y \ \ \ \text{ if} \ \ \  
\begin{cases} x = y, \text{ or} \\ 
\text{$x$ and $y$ are star twins and $\smallO(x) < \smallO(y)$, or}\\ 
\text{$x$ and $y$ are not star twins and } y \in A_x,\\
\end{cases}
\end{equation*}
where $\smallO:V(G) \to \{1, \hdots, |V(G)|\}$ is a fixed ordering of $V(G)$. The relation $\leq_A$ is integral to constructing even sets in $\R \setminus \beta$. Let $X_1 = \{x_1, \hdots, x_m\}$, and let $D_i = A_{x_i} \cap \R$ for $1 \leq i \leq m$. By Lemma \ref{lemma:partial-order}, $\leq_A$ is a partial order. Let $\ell: V(G) \to \{1, \hdots, |V(G)|\}$ be an ordering of $V(G)$ that corresponds to a linear extension of $\leq_A$ (which exists due to the order-extension principle \cite{order-extension}). So $\ell$ is a total order and for every $u, v \in V(G)$ with $u \neq v$, if $u \leq_A v$ then $\ell(u) < \ell(v)$. For each $D_i$, let $\ell_{x_i}: V(D_i) \to \{1, \hdots, |D_i|\}$ be an ordering of $V(D_i)$ such that for all $u, v \in D_i$, $\ell_{x_i}(u) < \ell_{x_i}(v)$ if and only if $\ell(u) < \ell(v)$.  
Now, we define $\delta^2$ sets $L_1, \hdots, L_{\delta^2}$, as follows: 
$$ L_i' = \bigcup_{x_j \in X_1} \ell_{x_j}^{-1}(i),$$ and 
$$L_i = L_i' \setminus (L_{i+1} \cup \hdots \cup L_{\delta^2}).$$ 
It follows from the definition of the sets $L_1, \hdots, L_{\delta^2}$ that for every $v \in \R \setminus \beta$, $v \in L_i$ if and only if $i = \max_{x_i \in X_1: v \in A_{x_i}}(\ell_{x_i}(v))$. Note also that $L_1, \hdots, L_{\delta^2}$ are disjoint and $\R \setminus \beta = L_1 \cup \hdots \cup L_{\delta^2}$. We call $(L_1, \hdots, L_{\delta^2})$ the {\em even set representation of $\R \setminus \beta$}. The {\em prefix of $L_i$}, denoted $\Pre(L_i)$, is $\Pre(L_i) = L_1 \cup \hdots \cup L_{i-1}$ for $i > 1$, and $\Pre(L_1) = \emptyset$. By the construction of $L_1, \hdots, L_{\delta^2}$, we have that for all $v \in V(G) \setminus \Pre(L_i)$ and all $x \in L_i$, $v \in A_{x}$ if and only if $v$ and $x$ are star twins. This  reproduces the conditions for the contradiction from Lemma \ref{lemma:T1T2_even}, and will play a key role in getting a contradiction in this case. 

In the remainder of this section, we prove that $(L_1, \hdots, L_{\delta^2})$ is a $\delta^2$-iterated even set of $G$. We need the following lemmas.

\begin{lemma}
Let $(L_1, \hdots, L_{\delta^2})$ be the even set representation of $\R \setminus \beta$. Let $u, v \in \R \setminus \beta$ be distinct vertices such that $u \leq_A v$, and let $i$ be such that $u \in L_i$. Then, $v \not \in L_i \cup \Pre(L_i)$. 
\label{lemma:xy_not_Li}
\end{lemma}
\begin{proof}
 Since $u \leq_A v$ and $u \neq v$, it follows that $v \in A_u$. Suppose $x_i \in X_1$ such that $u \in A_{x_i}$. By Lemma \ref{lemma:star_shields}, $A_u \setminus \{x_i\} \subseteq A_{x_i}$, so $v \in A_{x_i}$. Then, $\ell_{x_i}(v) > \ell_{x_i}(u)$ for all $x_i \in X_1$ such that $u \in A_{x_i}$, so $\max_{x_i \in X_1: v \in A_{x_i}}(\ell_{x_i}(v)) > \max_{x_i \in X_1: u \in A_{x_i}}(\ell_{x_i}(u))$.  It follows that $v \not \in L_i \cup \Pre(L_i)$.
\end{proof}

\begin{lemma}
Let $(L_1, \hdots, L_{\delta^2})$ be the even set representation of $\R \setminus \beta$ and let $D$ be a connected component of $\gamma \setminus \beta$. Then, for all $1 \leq i \leq \delta^2$, $|L_i \cap D| \leq 1$. 
\label{lemma:Li_cap_D}
\end{lemma}
\begin{proof}
Suppose $x, y \in L_i \cap D$. Since $D$ is a connected component of $\gamma \setminus \beta$, it follows by Lemma \ref{lemma:structure_of_R} that for all $x_1 \in X_1$, $x \in A_{x_1}$ if and only if $y \in A_{x_1}$. Assume by symmetry that $\ell(x) < \ell(y)$. Then, for all $x_j \in X_1$ such that $x, y \in A_{x_j}$, we have $\ell_{x_j}(x) < \ell_{x_j}(y)$. Therefore, $\max_{x_j \in X_1: x \in A_{x_j}}(\ell_{x_j}(x)) < \max_{x_j \in X_1: y \in A_{x_j}}(\ell_{x_j}(y))$. But $\max_{x_j \in X_1: x \in A_{x_j}}(\ell_{x_j}(x)) = \max_{x_j \in X_1: y \in A_{x_j}}(\ell_{x_j}(y)) = i$, a contradiction. 
\end{proof}

\begin{lemma}
Suppose $x \in X$ and $u \in V(G)$ such that $x \in A_u$. Then, $u$ and $x$ are star twins. 
\label{lemma:L1}
\end{lemma}
\begin{proof}
Since $x \in X$, it follows that $x$ is minimal with respect to $\leq_A$. Therefore, $x$ and $u$ are star twins and $\smallO(x) < \smallO(u)$. 
\end{proof}

\begin{lemma}
\label{lemma:L2}
Let $(X_1, X_2)$ be the bipartition of $\beta$. Suppose $x, y \in X_2$ and $u \in V(G)$ such that $u \in C_x$. Then, $y \not \in A_u$. 
\end{lemma}
\begin{proof}
Suppose $y \in A_u$. By Lemma \ref{lemma:L1}, it follows that $y$ and $u$ are star twins, so $u \in A_y$. But now $u \in A_y \cap C_x$, contradicting Lemma \ref{lemma:loosely_noncrossing}. 
\end{proof}

Now, we prove the main result of this section: $(L_1, \hdots, L_{\delta^2})$ is a $\delta^2$-iterated even set in $G \setminus \beta$. 

\begin{lemma}
Let $(L_1, \hdots, L_{\delta^2})$ be the even set representation of $\R \setminus \beta$. Let $1 \leq i \leq \delta^2$, let $x, y \in L_i$, and let $P$ be a path from $x$ to $y$ with $P^* \subseteq G \setminus (\beta \cup \Pre(L_i))$. Then, $P$ is even. 
\label{lemma:major_even_paths}
\end{lemma}
\begin{proof}
We may assume that $P^* \cap L_i = \emptyset$. 
Fix $1 \leq i \leq \delta^2$. Let $P = p_1 \dd \hdots \dd p_k$, with $p_1 = x$ and $p_k = y$, and suppose $P$ is odd. By Lemma \ref{lemma:Li_cap_D}, $|L_i \cap D'| \leq 1$ for every connected component $D'$ of $\gamma \setminus \beta$, so  $x$ and $y$ are in different connected components of $\gamma \setminus \beta$. It follows that $x$ and $y$ are an even pair in $\gamma \setminus \beta$, and $P \cap (G \setminus \gamma) \neq \emptyset$.

\vspace{-0.3cm}

\begin{equation}
\longbox{{\it There exist $u \in X_2$ and $p_q, p_r, p_s \in P \cap C_u$ with $q < s - 2$ and either $r = q+1$ or $r = s-1$, such that $\{u, p_q, p_r, p_s\}$ is a paw of $G$ with edge set $\{up_q, up_r, up_s, p_qp_r\}$ or $\{up_q, up_r, up_s, p_rp_s\}$.}}
\label{eq:there_is_edge}\tag{$7.11.1$}
\end{equation}
{\em Proof of \eqref{eq:there_is_edge}}:
Since $P \cap (G \setminus \gamma) \neq \emptyset$ and $G \setminus \gamma = \bigcup_{x_2 \in X_2} A_{x_2}$, we have $P \cap A_{x_2} \neq \emptyset$ for some $x_2 \in X_2$ with at least two non-adjacent neighbors in $P$ (in particular, the span of $x_2$ in $P$ is of length greater than one). By Lemma \ref{lemma:exists_a_paw}, if there exists $x_2 \in X_2$ such that the span of $x_2$ in $P$ is wide odd, then the claim holds. Therefore, we may assume that for every such $x_2$, the span of $x_2$ is even. Let $x_2 \in X_2$ be such that $P \cap A_{x_2} \neq \emptyset$, and let $p_i$ and $p_j$ be the neighbors of $x_2$ in $P$ closest to $x$ and $y$, respectively. As the span of $x_2$ is even, it follows that $p_i \dd P \dd p_j$ is an even subpath of $P$. 

 Since $p_i \dd P \dd p_j$ is an even subpath of $P$, and $P$ is odd, it follows that either $x \dd P \dd p_i$ is odd, or $y \dd P \dd p_j$ is odd. Up to symmetry, assume $x \dd P \dd p_i$ is odd, so in particular, $p_i \neq x$. Let $u_x \in X_2$ be such that $x \in C_{u_x}$ (which exists by definition of $\R$ and since $x \in \R \setminus \beta$). If $x$ is the unique neighbor of $u_x$ in $x \dd P \dd p_i$, then $u_x \dd x \dd P \dd p_i \dd x_2$ is an odd path, contradictiong Lemma \ref{lemma:T1T2_even}, so $u_x$ has a neighbor in $\{p_2, \hdots, p_i\}$. Let $z$ be the neighbor of $u_x$ closest to $p_i$. Then, $u_x \dd z \dd P \dd p_i \dd x_2$ is an even path from $u_x$ to $x_2$. It follows that $x \dd P \dd z$ is odd. If $u_x$ has wide odd span in $x \dd P \dd p_i \dd x_2$, then, by Lemma \ref{lemma:exists_a_paw}, there exists $p_q, p_{q+1}, p_r \in x \dd P \dd z$ with $r < q-1$ or $r > q+2$, such that $p_q, p_{q+1}, p_r \in C_{u_x}$. Therefore, we may assume $z$ is adjacent to $x$.

We note the following relations: 
\begin{enumerate}[(i)]
\itemsep -0.2em
    \item $x \in C_{u_x}$. (By choice of $u_x$).
    \item $x_2 \in B_{u_x}$, $u_x \in B_{x_2}$. (Since $x_2, u_x \in X_2$).
    \item $z, x \in C_{u_x}$. (Since $x \dd P \dd p_i \dd x_2$ is a path with $x, x_2 \in B_{u_x} \cup C_{u_x}$, and $x \dd z$ is the span of $u_x$ in $x \dd P \dd p_i \dd x_2$, it follows that $x, z \in C_{u_x}$).
    \item $x_2 \not \in A_{z}$, $x_2 \not \in A_x$. (By Lemma \ref{lemma:L2}, since $x, z \in C_{u_x}$, it follows that $x_2 \not \in A_x$ and $x_2 \not \in A_z$).
    \item $u_x \not \in A_x$, $u_x \not \in A_z$. (Suppose $u_x \in A_x$. By Lemma \ref{lemma:L1}, it follows that $u_x$ and $x$ are star twins. But $x \in C_{u_x}$, a contradiction. So $u_x \not \in A_x$, and similarly, $u_x \not \in A_z$.). 
\end{enumerate}

Now, $xu_xzt$ is a paw of $G$, where $t = x_2$ if $z = p_i$, and $t = p_3$ otherwise. Suppose $t \in A_x$. By (iv), $x_2 \not \in A_x$, so it follows that $t = p_3$ and $p_3 \in A_x$. But then $p_3 \dd P \dd p_i \dd x_2$ is a path from $p_3$ to $x_2$ anticomplete to $x$, so $x_2 \in A_x$, a contradiction. Therefore, $t \not \in A_x$. Similarly, $t \not \in A_z$ and $t \not \in A_{u_x}$. 

By Theorem \ref{thm:paws_are_forcers}, either $x$ breaks $\{u_x, t\}$, or $u_x$ breaks $\{x, t\}$, or $z$ breaks $\{x, u_x, t\}$.  Suppose $x$ breaks $\{u_x, t\}$, so $A_x \cap \{u_x, t\} \neq \emptyset$. But $t \not \in A_x$, and by (v), $u_x \not \in A_x$, a contradiction. Now, suppose $u_x$ breaks $\{x, t\}$, so $A_{u_x} \cap \{x, t\} \neq \emptyset$. But $t \not \in A_{u_x}$, and by (i), $x \not \in A_{u_x}$, a contradiction. So $z$ breaks $\{x, u_x, t\}$. Suppose $z$ breaks $\{u_x, t\}$, so $A_z \cap \{u_x, t\} \neq \emptyset$. But $t \not \in A_z$, and by (v), $u_x \not \in A_z$, a contradiction. Therefore, $z$ breaks $\{x, t\}$. Since $t \not \in A_z$, it follows that $x \in A_z$. 

Since $z \in C_{u_x} \cap (G \setminus (\beta \cup \Pre(L_i)))$, it follows that $z \in \R \setminus \beta$. Since $x$ and $z$ are adjacent, they are in the same connected component of $\gamma \setminus \beta$. Since $x \in L_i$ and $z \not \in \Pre(L_i)$, it follows that $\ell(z) > \ell(x)$. If $x$ and $z$ are not star twins, then since $x \in A_z$, it follows that $z \leq_A x$, and hence $\ell(z) < \ell(x)$, a contradiction. So $x$ and $z$ are star twins. Now, $t$ is adjacent to $z$ and $t \not \in A_z$, so $t \in C_z$. Since $z$ and $x$ are star twins, it follows that $t \in C_x$. But $t$ and $x$ are not adjacent, a contradiction. 
This proves~\eqref{eq:there_is_edge}.\\

Suppose $x \in A_y$. Then, either $x \leq_A y$ or $y \leq_A x$, so by Lemma \ref{lemma:xy_not_Li}, either $x \not \in L_i$ or $y \not \in L_i$, a contradiction. Therefore, $x \in B_y \cup C_y$, and by symmetry, $y \in B_x \cup C_x$. Let $u$ and $p_2$ be as in \eqref{eq:there_is_edge}. Since $p_r \in C_{u} \setminus \beta$ and $u \in X_2$, it follows that $p_r \in \R \setminus \beta$. Let $j$ be such that $p_r \in L_j$. Since $P^* \subseteq G \setminus (\beta \cup \Pre(L_i))$ and $P^* \cap L_i = \emptyset$, it follows that $j > i$. Let $x_1 \in X_1$ be such that $p_r \in A_{x_1}$. By Lemma \ref{lemma:key_lemma_paw_shield}, either $p_r \leq_A x$ or $p_r \leq_A y$. We may assume by symmetry that $p_r \leq_A x$. But by Lemma \ref{lemma:xy_not_Li}, it follows that $x \not \in \Pre(L_j)$, a contradiction.  
\end{proof}

\section{Even set separators and $(k, c, d, m)$-tame graphs}
\label{sec:lastsection}

In this section, we prove Theorem \ref{thm:pawfriendlyevenseparator}. First, we use the results of Section \ref{sec:even_set_separators} to prove that paw-friendly graphs with no bounded balanced separator have even set separators. 

\begin{lemma}
  Let $c \in [\frac{1}{2}, 1)$ and let $d, \delta$ be positive integers with $d \geq \delta + 3$. Let $G$ be a connected paw-friendly graph on $n$ vertices with maximum degree $\delta$, let $w: V(G) \to [0, 1]$ be a weight function on $V(G)$ with $w(G) = 1$, and suppose $G$ has no $d$-bounded $(w, c)$-balanced separator. Then, $G$ has a $(w, \delta^2+2, c, \delta + 1)$-even set separator $L $ and one can find an even set representation of $L$ in time $\mathcal{O}(n^3)$. In particular, $G$ is
    $(\delta^2+2,c,\delta+1,3)$-tame.
\label{lemma:first_final}
\end{lemma}
\begin{proof}
For every vertex $v \in V(G)$, one can compute the canonical star separation $S_v = (A_v, C_v, B_v)$ in linear time, so the canonical star separations for every vertex $v \in V(G)$ can be computed in time $\mathcal{O}(n^2 + nm)$. The partial order $\leq_A$ can be computed on $V(G)$ in time $\mathcal{O}(n^3)$. Let $X\subseteq V(G)$ be the star covering of $G$, i.e., the set of minimal vertices of $G$ with respect to $\leq_A$. Let $\S$ be the $\smallO$-star covering sequence of $G$ and let $\beta$ be the star-free bag of $G$. By Lemma \ref{lemma:dim_is_2}, $X$ is bipartite. The bipartition $(X_1, X_2)$ of $X$ can be obtained in linear time. Recall that $\beta = X$, and let $\R = \beta \cup (\bigcup_{x_2 \in X_2} C_{x_2})$. Observe that given $\S$, $\R$ can be computed in linear time.

A linear extension $\ell$ of $\leq_A$ can be computed in linear time by running a depth-first search algorithm on a representation of the partial order as a directed acyclic graph. The functions $\ell_{x_i}$ (as defined in Section \ref{sec:even_sets_Rminusbeta}) for all $x_i \in X_1$ can be obtained from the linear extension in linear time. Then, the sets $L_1, \hdots, L_{\delta^2}$ can be computed in linear time, and $\R \setminus \beta = L_1 \cup \hdots \cup L_{\delta^2}$. Let  $L = (X_1, X_2, L_1, \hdots, L_{\delta^2})$. By Lemma \ref{lemma:structure_of_R}, it follows that for every connected component $D$ of $G \setminus L$, $D \subseteq A_v$ for some $v \in \beta$ and $N(D) \subseteq N[v]$, and hence $|N(D)| \leq \delta + 1$. Furthermore, by Lemma \ref{lemma:skewed}, $w(A_v) \leq 1-c$ and so, since $c \geq \frac{1}{2}$, $w(D) \leq c$. By Lemma \ref{lemma:T1T2_even}, it follows that $X_1$ and $X_2$ are even sets in $G$, and by Lemma \ref{lemma:major_even_paths}, $(L_1, \hdots, L_{\delta^2})$ is a $\delta^2$-iterated even set in $G \setminus \beta$. Therefore, $(X_1, X_2, L_1, \hdots, L_{\delta^2})$ is a $(\delta^2 + 2)$-iterated even set in $G$, and hence $L$ is a $(w, \delta^2 + 2, c, \delta + 1)$-even set separator of $G$ which can be computed in time $\mathcal{O}(n^3)$. 
\end{proof}

Now we prove Theorem~\ref{thm:pawfriendlyevenseparator} that we restate.
\pawfriendlyevenseparator*
    \label{thm:bigthm_pawfriendly}
\begin{proof}
Let $w$ be a uniform weight function on $V(G)$, so $w(G) = 1$. For every $v \in V(G)$, we check whether every component $D$ of $G \setminus N^{\delta + 3}[v]$ has $w(D) \leq c$, which can be done in $\mathcal{O}(n^{2} + nm)$ time. Suppose there exists $v \in V(G)$ such that $w(D) \leq c$ for every component $D$ of $G \setminus N^{\delta + 3}[v]$ (so $G$ has a $(\delta + 3)$-bounded $(w, c)$-balanced separator). Let $X = N^{\delta + 3}[v]$, so $|X| \leq 1 + \delta + \hdots + \delta^{\delta + 3}$. Let $X = \{x_1, \hdots, x_{\ell}\}$. Then, $\{\{x_1\}, \hdots, \{x_{\ell}\}\}$ is a $(w, \ell, c, \ell)$-even set separator of $G$, found in time $\mathcal{O}(n^{3})$. 

Now, suppose $G$ does not have a $(\delta + 3)$-bounded $(w, c)$-balanced separator. By Lemma \ref{lemma:first_final}, $G$ has a $(w, \delta^2 + 2, c, \delta + 1)$-even set separator found in time $\mathcal{O}(n^3)$. Therefore, $G$ is $(1 + \delta + \hdots + \delta^{\delta + 3}, c, 1 + \delta + \hdots + \delta^{\delta + 3}, 3)$-tame. 
\end{proof}


\section{Acknowledgments}
The authors would like to thank Alexander Schrijver, Bruce Shepherd, and Richard Santiago Torres for helpful discussions about submodular functions.


\end{document}